\documentclass[11pt]{amsart}
\usepackage{amsfonts,epsfig}
\usepackage{newlfont,amsfonts,amssymb,amsmath,amsthm,amsgen,amscd,datetime,enumerate}
\usepackage{multicol,picinpar}
\usepackage{amsaddr}

\usepackage{euscript,color}

\newcommand{\edremoved}[2][\empty]{\marg{}{#2}}
\newcommand{\edrevised}[2][\empty]{\marg{}{#2}}
\newcommand{\ednew}[2][\empty]{\marg{}{#2}}

\usepackage{hyperref}
\hyperbaseurl{.}

\newcommand{\marg}[1]{\marginpar{\small {\tt  #1 }}}

\newcommand{\hess}{\mathrm{Hess}}

\newcommand{\Z}{\ensuremath{\mathbb{Z}}}

\renewcommand{\O}{\ensuremath{\mathrm{O}}}
\newcommand{\R}{\ensuremath{\mathbb{R}}}
\renewcommand{\SS}{\ensuremath{\mathbb{S}}}

\newcommand{\bmat}{\begin{pmatrix}}
\newcommand{\emat}{\end{pmatrix}}

\newcommand{\e}{\mathrm{e}}

\newcommand{\bcase}{\begin{case}}
\newcommand{\ecase}{\end{case}}

\newcommand{\bclaim}{\begin{claim}}
\newcommand{\eclaim}{\end{claim}}

\newcommand{\bstep}{\begin{step}}
\newcommand{\estep}{\end{step}}

\newcommand{\bhlem}{\begin{hlem}}
\newcommand{\ehlem}{\end{hlem}}

\newcommand{\bleer}{\begin{leer}}
\newcommand{\eleer}{\end{leer}}
\newcommand{\bde}{\begin{de}}
\newcommand{\ede}{\end{de}}

\newcommand{\bs}{\begin{satz}}
\newcommand{\es}{\end{satz}}
\newcommand{\btheo}{\begin{theo}}
\newcommand{\etheo}{\end{theo}}
\newcommand{\bfolg}{\begin{folg}}
\newcommand{\efolg}{\end{folg}}
\newcommand{\blem}{\begin{lem}}
\newcommand{\elem}{\end{lem}}
\newcommand{\bnote}{\begin{note}}
\newcommand{\enote}{\end{note}}
\newcommand{\bprf}{\begin{proof}}
\newcommand{\eprf}{\end{proof}}
\newcommand{\bd}{\begin{displaymath}}
\newcommand{\ed}{\end{displaymath}}
\newcommand{\be}{\begin{eqnarray*}}
\newcommand{\ee}{\end{eqnarray*}}
\newcommand{\eeqa}{\end{eqnarray}}
\newcommand{\beqa}{\begin{eqnarray}}
\newcommand{\bi}{\begin{itemize}}
\newcommand{\ei}{\end{itemize}}
\newcommand{\bnum}{\begin{enumerate}}
\newcommand{\enum}{\end{enumerate}}

\newcommand{\ve}{\varepsilon}

\newcommand{\lra}{\longrightarrow}

\newcommand{\sub}{\subset}

\newcommand{\beq}{\begin{equation}}
\newcommand{\eeq}{\end{equation}}

\newcommand{\rr}{\mathbb{R}}

\newcommand{\M}{{\cal M}}

\newcommand{\vf}{\varphi}
\newcommand{\earr}{\end{array}\]}
\newcommand{\barr}{\[\begin{array}}
\newcommand{\bvec}{\left(\begin{array}{c}}
\newcommand{\evec}{\end{array}\right)}
\newcommand{\sumk}{\sum_{k=1}^n}
\newcommand{\sumi}{\sum_{i=1}^n}

\newcommand{\sumij}{\sum_{i,j=1}^n}
\newcommand{\sumkl}{\sum_{k,l=1}^n}

\newcommand{\g}{\mathfrak{g}}

\newcommand{\Hol}{\mathrm{Hol}}

\newcommand{\rrn}{\mathbb{R}^n}

\renewcommand{\b}{\flat}

\newcommand{\w}{\omega}

\newcommand{\s}{\sigma}
\newcommand{\del}{\partial}

\newcommand{\bbem}{\begin{bem}}
\newcommand{\ebem}{\end{bem}}
\newcommand{\bbez}{\begin{bez}}
\newcommand{\ebez}{\end{bez}}
\newcommand{\bbsp}{\begin{bsp}}
\newcommand{\ebsp}{\end{bsp}}

\newcommand{\Ker}{\mathrm{Ker }}

\newcommand{\grad}{\mathrm{grad}}
\newcommand{\wt}{\widetilde}
\newcommand{\tnab}{\wt{\nabla}}

\newcommand{\tem}{\widetilde{\cal M}}
\newcommand{\ten}{\widetilde{\cal N}}
\newcommand{\tg}{\widetilde{g}}

\renewcommand{\=}{\equiv}

\newcommand{\T}{\mathbb{T}}

\newcommand{\belabel}[1]{\begin{equation}\label{#1}}

\theoremstyle{definition}
\newtheorem{de}{Definition}
\newtheorem{bem}{Remark}
\newtheorem{bez}{Notation}
\newtheorem{bsp}{Example}

\newtheorem*{bsp*}{Example}
\newtheorem*{def*}{Definition}
\theoremstyle{plain}
\newtheorem{lem}{Lemma}
\newtheorem*{lem*}{Lemma}
\newtheorem{satz}{Proposition}
\newtheorem{folg}{Corollary}
\newtheorem{theo}{Theorem}

\begin{document}
\bibliographystyle{abbrv}

\title
{Completeness of compact  Lorentzian manifolds with Abelian holonomy}
\author[Thomas Leistner]{Thomas Leistner}
\address[Leistner, corresponding author]{\small School of Mathematical Sciences, University of Adelaide, SA 5005, Australia.\\ E-mail: thomas.leistner@adelaide.edu.au, telephone: +61 (0)8 83136401,\linebreak
fax: +61 (0)8 83133696 (Corresponding author)
} 

\author[Daniel Schliebner]{Daniel Schliebner} \address{\small Humboldt-Universit\"{a}t zu Berlin,
Institut f\"{u}r Mathematik, Rudower Chaussee~25,\linebreak 12489~Berlin, Germany.
E-mail: schliebn@mathematik.hu-berlin.de}

\thanks{This work was supported by the Group of Eight Australia and the German Academic Exchange
Service through the Go8 - DAAD Joint Research Co-operation
 Scheme.  
 The
first author acknowledges support from the Australian Research
Council via the grants FT110100429 and DP120104582. The second author is funded by the Berlin Mathematical School.}

\begin{abstract} We address the problem of finding  conditions under which a compact Lorentzian manifold is geodesically complete, a property, which always holds for compact Riemannian manifolds.
It is known that a compact Lorentzian manifold is geodesically complete  if it  is 
 homogeneous, or has constant curvature, or  admits a time-like conformal vector field.
We consider certain Lorentzian manifolds with Abelian holonomy, which are locally modelled by the so called pp-waves, and which, in general,  do not satisfy any of the above conditions. 
We show that compact pp-waves are universally covered by  a vector space, determine the metric on the universal cover, and prove that they are geodesically complete. Using this, we show that every Ricci-flat compact pp-wave is  a plane wave.
\end{abstract}

\subjclass[2010]{Primary 53C50; Secondary 53C29, 53C12,  53C22}
\keywords{Lorentzian manifolds, geodesic completeness, special holonomy, pp-waves}

\maketitle



\section{Introduction and statement of results}

An important class of Lorentzian manifolds are those with {\em special holonomy}. Following the terminology in Riemannian geometry, these are Lorentzian manifolds for which 
the connected holonomy group acts indecomposably, i.e., the manifold does not
 locally decompose into a product but  the holonomy   still is reduced from the full orthogonal group. In contrast to the Riemannian situation, the latter prevents the holonomy group from acting irreducibly and equips the manifold with a  bundle of tangent null lines which is invariant under parallel transport. We will study the geodesic completeness for a certain type of Lorentzian manifolds with special holonomy, namely those with {\em Abelian holonomy}.
 A semi-Riemannian manifold is {\em geodesically complete}, or for short {\em complete}, if all maximal geodesics are defined on $\rr$. 
 
After Berger's classification of connected irreducibly acting {\em Riemannian} holonomy groups \cite{berger55}, the quest for   {\em complete} or {\em compact}  Riemannian manifolds with holonomy groups from Berger's list produced some of the highlights of modern differential geometry, for example, Yau's proof of Calabi's conjecture or Joyce's construction of compact manifolds with exceptional holonomy (see \cite{joyce00} for a full account of  such results).
For Lorentzian manifolds
the classification of the connected components of indecomposable {\em Lorentzian}  holonomy groups was obtained in \cite{leistnerjdg} based on results in \cite{bb-ike93}. Moreover,  in \cite{galaev05} a construction method for  Lorentzian metrics was developed, which showed that indeed all possible groups can be realised as holonomy groups.
 A survey about the general classification is given in 
 \cite{galaev-leistner-esi}. 
 Furthermore, 
 a first attempt to investigate the {\em full holonomy group} and  global properties of the manifolds, such as global hyperbolicity, was made in \cite{baum-laerz-leistner12}. Compact Lorentzian manifolds with special holonomy  have also been studied in \cite{derdzinski-roter08,derdzinski-roter10}, in \cite{laerz-diss},  and in \cite{schliebner12,Schliebner15}. 
 
 Since there are no proper connected irreducible subgroups of the Lorentz group \cite{olmos-discala01}, and since the construction of Lorentzian manifolds with special  holonomy in some parts relies on the existence  results in Riemannian geometry,   finding Lorentzian manifolds with prescribed holonomy is  easier than in the Riemannian context. However, the relation between  compact and  complete examples is more subtle, since --- in sharp contrast to the Riemannian world --- compact Lorentzian manifolds do not have to be complete. 
  The standard example of this phenomenon is the Clifton-Pohl torus, which is  compact, but geodesically incomplete \cite[Example 7.16]{oneill83}. Hence, finding compact Lorentzian manifolds with special holonomy does not automatically provide geodesically complete examples.
  
 The question whether a compact Lorentzian manifold is complete is  classical in global Lorentzian geometry. 
Under some strong assumptions,  a compact Lorentzian manifold is complete, for example,  if it is flat \cite{carriere89}\footnote{In fact,  in \cite{carriere89} Carri\`ere proved a much more general result for affine manifolds. A direct proof for the flat case was given in   \cite{yurtsever92}. However, this proof has gaps as it was pointed out in \cite{romero-sanchez93}.},   has  constant curvature \cite{klingler96}, or if it is  homogeneous. In fact,  Marsden proved in \cite{marsden72} that any compact homogeneous semi-Riemannian manifold is complete. Moreover,  compact, {\em locally} homogeneous $3$-dimensional Lorentzian manifolds  are complete \cite{zeghib-dumitrscu10}. Finally, in \cite{romero-sanchez95} it was shown that compact Lorentzian manifolds with a {\em time-like conformal Killing vector field} are  complete (see also  \cite{kamishima93} or  \cite{romero-sanchez94semi, romero-sanchez94}). 
We are going to consider
Lorentzian manifolds with Abelian holonomy, the  so-called {\em pp-waves}. In general they do not satisfy any of these conditions, they are not locally homogeneous, not of constant curvature and do not admit a time-like conformal Killing vector field.
\bde\label{ppdef}
A Lorentzian manifold $(\cal M,g)$ 
is 
 called {\em pp-wave}\footnote{In the following we will consider compact manifolds of this type. We are aware that for compact manifolds the term {\em wave} might not be appropriate, but we 
use this term since it is established in the literature for manifolds with the given curvature 
properties. Later we will see that an appropriate name would be {\em screen flat}, but this term has 
other obvious problems.}  if it admits a
 global parallel null vector field $V \in \Gamma(T\cal M)$, i.e., $V\not=0$, $g(V,V)=0$ and $\nabla V=0$, and if  its curvature tensor $R$   satisfies
\begin{equation}\label{screen-flat}
	R(U,W)=0,\ \text{for all} \ \ U,W \in V^\bot.
\end{equation}
\ede
Here $\nabla$ denotes the Levi-Civita connection of $g$ and $R$  the curvature tensor of $\nabla$, $R\in \Lambda^2T^*\M\otimes \mathrm{End}(T\M)$ of $(\cal M,g)$. 
A  pp-wave metric locally depends only on one function: for a {pp-wave} $(\cal M, g)$  there are local coordinates $\left(\cal U,(u,v,x^1, \ldots , x^n)\right)$ such that 
\belabel{ppcoord}
g|_{\cal U}=2du (dv + Hdu )+  \delta_{ij}dx^idx^j,
\end{equation}
where $\dim \M=n+2$ and $H = H(u,x^1, \ldots ,x^n)$ is a smooth function on the coordinate patch $\cal U$  not depending on $v$.
If $\cal M=\rr^{n+2}$ and $g$ is {\em globally} of the form \eqref{ppcoord} we call $(\M,g)$ a 
 {\em  pp-wave  in standard form} or a {\em standard pp-wave}. Four-dimensional standard pp-waves were  discovered  by Brinkmann in the context of conformal geometry \cite{brinkmann25}, and then played an important role in general relativity  (e.g., see \cite{ehlers-kundt62}, where also  the name {\em pp-wave} for {\em plane fronted with parallel rays} was introduced).  More recently, as manifolds with a maximal number of parallel spinors, higher dimensional pp-waves appeared in supergravity theories, e.g. in \cite{hull84pp}, and there is now a vast physics literature on them.
  
 The existence of a parallel null vector field and the curvature condition imply that pp-waves have their holonomy contained in the Abelian ideal $\rr^n$ of the stabiliser $\O(n)\ltimes \rrn$ of a null vector. Moreover, under a mild genericity condition on the function $H$, their holonomy is {\em equal} to $\rrn$ and hence acts indecomposably.
Our  results about {\em compact} pp-waves can be summarised in   two theorems:

\begingroup
\renewcommand\thetheo{\Alph{theo}}

	\begin{theo}
		\label{MainTheo2}
The  universal cover of an $(n+2)$-dimensional compact pp-wave is globally isometric to a standard  pp-wave	\[
			\left(  \R^{n+2}, \ \ g^H = 
			2dudv + 2H(u,x^1, \ldots, x^n)du^2 + \delta_{ij}dx^idx^j\right).
\]
	Under this isometry, the lift of the parallel null vector field  is mapped to $\frac{\del}{\del v}$.
	\end{theo}
	
	The proof in Section \ref{thmAsection} uses the so-called {\em screen bundle} $\Sigma=V^\bot/V\to \M$	and the induced {\em screen distributions} (as described in Section~\ref{screensection}). They can be used to define 
	Riemannian metrics on the leaves of $V^\bot$, which are flat in case of pp-waves. This yields  a  detailed description of the universal cover of pp-waves in Section~\ref{coversection}. Then, results by  Candela et al.~\cite{candela-flores-sanchez03} about the completeness of certain
	{\em non-compact} Lorentzian manifolds, which apply to  pp-waves in standard form, enable us 
	in Section~\ref{thmBsection}
	 to prove

	\begin{theo}
		\label{MainTheo3}
		Every compact pp-wave $(\cal M, g)$ is geodesically complete.
	\end{theo}

\endgroup
As a consequence we obtain that the examples of compact pp-waves we give below, are also  geodesically complete examples of  Lorentzian manifolds of special holonomy.


Theorem \ref{MainTheo3} is somewhat surprising when recalling that Ehlers and Kundt  posed the following problem
 \cite[Section 2-5.7]{ehlers-kundt62}\footnote{We thank Wolfgang Globke for pointing  us to this reference.}
:
 \begin{quote}
{\em ``Prove the plane waves to be the only $g$-complete pp-waves, no matter which topology one chooses.''}
\end{quote}
The {\em plane waves} mentioned in the problem are a  special class of pp-waves:
\bde\label{planedef}
A   pp-wave $(\M,g)$ with parallel null vector field $V$ is a {\em plane wave} if 
 \belabel{planewave}
  \nabla R = V^\flat \otimes Q,\end{equation}
   where $R$ is the curvature tensor of $(\M,\g)$, $Q$ is a $(0,4)$-tensor field,  and  $V^\flat:=g(V,.)$. 
   \ede
   For a plane wave, the function $H$ in the local form~\eqref{ppcoord}
 of the metric is of the form
   \belabel{spw}
   H(u,x^1, \ldots x^n)=a_{ij}(u)x^ix^j, \text{ with }a_{ij}=a_{ji}\in C^\infty(\rr).
   \end{equation}
       This can be used to show that plane waves (in standard form, i.e., with $\M=\rr^{n+2}$ and $g=g^H$ with $H$ as in \eqref{spw})  are always geodesically complete (\cite[Proposition 3.5]{candela-flores-sanchez03}, we review this result in Section~\ref{seccomplete}). 
    Our Theorem~\ref{MainTheo3} shows that {\em any compact} pp-wave is complete, even if it is not a plane wave.
    
\bbsp
    \label{torusexample}
Let $\eta$ be the flat metric on the $n$-torus $\mathbb{T}^n$ and $H\in C^\infty(\mathbb{T}^n)$ a smooth function on $\mathbb{T}^n$. On $\cal M:=\mathbb{T}^2\times \mathbb{T}^n$ we consider the Lorentzian metric
\belabel{example}g^H=2d\theta d\vf + 2H d\theta^2 +\eta,\end{equation}
where $d\theta $ and $d\vf$ is the standard coframe on $\mathbb{T}^2$. This metric is a complete pp-wave metric on the torus $\mathbb{T}^{n+2}$, and one can choose $H$ in a way that it is not a plane wave. Indeed, computing $\nabla R$ shows that for any function $H$ with non-vanishing third partial derivatives with respect to the $x^i$-coordinates, the equality \eqref{planewave} is violated.
More examples are given in \cite{laerz-diss} and \cite{baum-laerz-leistner12}, and in our Example~\ref{bundleexample}.
\ebsp
However,  this example is not in contradiction to the claim in the Ehlers-Kundt problem because there
pp-waves  are understood to be solutions of the Einstein {\em vacuum} field equations and hence, in addition to Definition ~\ref{ppdef}, are assumed to be Ricci flat. 
 But the metric \eqref{example} is Ricci flat if and only if $H$ is harmonic with respect to the flat metric on the torus, which forces $H$ to be constant and   $g^H$ to be flat.
In fact, Theorem~\ref{MainTheo2} and results in Section \ref{thmBsection} allow us to generalise this observation.
 \bfolg\label{ricfolg}
Every compact Ricci-flat pp-wave is a plane wave.
\efolg
 This solves the Ehlers-Kundt problem in case of compact manifolds. 
 We should mention that, using their results in \cite{candela-flores-sanchez03},  another partial answer to the Ehlers-Kundt problem is given in     \cite[Theorem 4]{flores-sanchez06}. We describe this in our Section~\ref{seccomplete} and provide more examples of (non-compact) complete pp-waves that are not plane waves with our
Lemma~\ref{completelemma} and the corresponding Remark~\ref{EKremark} in Section~\ref{compactsection}.
 
Motivated by Corollary \ref{ricfolg}, in Section \ref{thmBsection} we apply Theorems \ref{MainTheo2} and \ref{MainTheo3}  to compact plane waves and conclude not only that they are complete but also  covered by a  plane wave in standard form. This
and results about compact plane waves with parallel Weyl tensor 
in \cite{derdzinski-roter08,derdzinski-roter10},
 suggest the problem of studying  compact quotients of plane waves, however, such investigations are beyond the scope of this paper.

\ednew{Theorem \ref{MainTheo3} has an interesting consequence for another special case, namely for compact indecomposable Lorentzian  locally symmetric spaces. A semi-Riemannian manifold is {\em decomposable} if it is locally a product, i.e., each point admits a neighbourhood on which the metric is a product metric. Otherwise the manifold is {\em indecomposable}.
  It is known that an indecomposable,  simply connected Lorentzian symmetric space is either  a Cahen-Wallach space or has a semisimple transvection group \cite[Theorems~2 and~3]{cahen-wallach70}, in which case it has constant sectional curvature \cite{berger57,olmos-discala01}.  Then,  for compact {\em locally} symmetric spaces, i.e., with $\nabla R=0$,  our results and those of \cite{klingler96} imply:
\bfolg\label{symcol}
An indecomposable, compact locally symmetric Lorentzian manifold is geodesically complete. As a consequence, it is  a  quotient by a lattice of either the universal cover of the odd-dimensional anti-de Sitter space or of  a Cahen-Wallach space.
\efolg
The proof of the geodesic completeness is given in  Section \ref{symsec}.
To obtain the further consequence in this corollary, recall that neither de Sitter spaces \cite{CalabiMarkus62} nor 
even-dimensional universal anti-de Sitter spaces have  compact quotients \cite{Kulkarni81}, see also \cite{wolf67}.
For lattices in the isometry groups of Cahen-Wallach spaces see \cite{MedinaRevoy85,Fischer13,KathOlbrich15}.}

Of course, it would be interesting to study geodesic completeness for the larger class of Lorentzian manifolds with special holonomy, i.e., with parallel null vector field, or more generally, with parallel null line bundle. We believe that some of our methods can be generalised to this setting, although such a generalisation is not straightforward
 as our Example~\ref{incompex} on page~\pageref{incompex} shows.
For instance, an interesting question is, whether  the manifolds constructed  in    
 \cite{galaev05} to realise all possible (connected) holonomy groups are complete. These problems will be subject to  further research.
 
\subsection*{Acknowledgements} We would like to thank Helga Baum and Miguel S\'{a}nchez  for helpful discussions and comments on the first draft of the paper. We also thank the the referees for valuable comments \ednew{and Ines Kath for alerting us to the implications our result has for locally symmetric spaces.}

\section{Screen distributions and pp-waves}
\label{screensection}

\subsection{The screen bundle  and  screen distributions}

Most of what follows in  this section immediately generalises to 
 general (time-orientable) Lorentzian manifolds with special holonomy,
i.e, with parallel null line bundle in the tangent bundle, \edrevised{see also \cite[Proposition 2]{baum-laerz-leistner12} for equivalent conditions}, 
but for our purposes we will restrict ourselves  to Lorentzian manifolds with {\em global parallel} null vector field
\edrevised{$V$, i.e., with
$g(V,V)=0$ and $\nabla V=0$, where $\nabla$ the Levi-Civita connection of $g$.
 In accordance with \cite{BlancoSanchezSenovilla13} we call such a manifold a {\em Brinkmann space}, after \cite{brinkmann25}, and we will from now on denote the parallel null vector field by $V$.}
Let 
 $(\M,g)$ be a Brinkmann space of 
 of dimension $(n+2)$.
 Since the vector field $V$ is parallel, also   the distribution of hyperplanes 
 \[V^\bot=\{X\in \Gamma(T\M) \mid g(X,V)=0  \}\] 
 is parallel in the  sense that it is invariant under parallel transport. Hence, we have a filtration of $T\M$ into parallel distributions
 \[
\rr\cdot V\subset V^\bot \subset T\M.\] 
 In particular, both distributions are involutive and define foliations $\cal V$ and $\cal V^\bot$ of the manifold $\cal M$. 
Locally, a Brinkmann space
 $(\cal M,g)$ 
 coordinates $(\cal U,\vf=(u,v,x^1, \ldots , x^n))$ such that the metric is given as
\belabel{coord}
g|_{\cal U}=2du (dv + Hdu +\mu_i dx^i)+  \hat{g}_{ij}dx^idx^j
\end{equation}
where $H$, 
$\hat{g}_{ij}=\hat{g}_{ij}(u,x^1, \ldots , x^n)$ and $\mu_i =\mu_i (u,x^1,\ldots , x^n)$ are smooth 
functions on $\cal U$, not depending on $v$.
In these coordinates, the parallel null vector field  $V|_{\cal U}$ is given by $\del_v$ and the leaves of $V^\bot|_{\cal U}$ are given by $u \= $ constant. 
%
%

 The above filtration   defines 
a vector bundle, the {\em screen bundle}
\[
\Sigma:=V^\bot/V \lra \cal M,\]
which is equipped with a positive definite metric induced by $g$,
\[g^\Sigma([X],[Y]):=g(X,Y)
,\] and a covariant derivative $\nabla^\Sigma$ induced by the Levi-Civita connection $\nabla $ of $g$,
\[\nabla^\Sigma_X[Y]=[\nabla _XY].
\] 

\bde
Let $(\cal M,g)$ be a \edrevised{Brinkmann space} of  dimension $n+2>2$.
A {\em screen distribution $\SS$} is a subbundle of $V^\bot$ of rank $n$ on which the metric $g$ is 
non-degenerate. 
A null vector field $Z$ such that $g(V,Z)\equiv 1$ is called a {\em screen vector field}.
\ede
Every screen vector field defines a screen distribution 
via $\SS:=V^\bot\cap Z^\bot$, where by $\bot$ we  denote the orthogonal space with respect to the metric $g$. 
In fact,  there is  a one-to-one correspondence between
screen distributions $\SS$,
 screen vector fields $Z$,
 and null one-forms $\zeta$ such that $\zeta(V)=1$:
\[
Z\longmapsto \SS=V^\bot\cap Z^\bot=\Ker(V^\flat, \zeta)\longmapsto\zeta=Z^\flat
\]where  $Z^\flat=g(Z,.)$ and $\Ker (.)$ is the annihilator of the one-forms in the argument. 

By partition of unity, $\cal M$ always admits a screen distribution which is a non-canonical splitting of the exact sequence
$0 \to \rr\cdot  V\to V^\bot \to \Sigma \to 0$. 
\edrevised{
Since 
 $\SS^{\bot}$ has rank $2$, 
the vector field 
 $V$   gives a unique global
 section $Z$ to the null-cones of $\SS^{\bot}$ with $g(V,Z)=1$, and thus} defining a  screen vector field. \edremoved{{\cite[Proposition 2]{baum-laerz-leistner12}}}

In the following we will make use of the Riemannian metric $h=h^\SS$ defined by a screen distribution 
$\SS$ or, equivalently, by a screen vector field $Z$ via
\begin{equation}\label{h}
h(V,.):=g(Z,.),\  \ h(Z,.):=g(V,.),\  h(X,.):=g(X,.) \text{ for }X\in\SS.
\end{equation}
and extension by linearity. 


\subsection{Horizontal and involutive screen distributions}

In the following we will call
a screen distribution {\em horizontal} if for every $p\in M$ exists a neighborhood $\cal U$ and local frame fields 
$S_1, \ldots, S_n \in \Gamma (\SS|_{\cal U})$ of $\SS$ such that 
\begin{equation}\label{horizontal}
[V,S_i]\in \Gamma (\SS|_{\cal U}),
\end{equation} 
and {\em involutive} if 
\begin{equation}\label{involutive}
[S_j,S_i]\in \Gamma (\SS|_{\cal U}).
\end{equation} 
Note that both conditions are independent of the chosen frame fields for $\SS$.
Note also that  the coordinates in \eqref{coord} provide us with a {\em local} horizontal screen distribution spanned by $\del_1+\mu_1\del_v,\ldots, \del_n+\mu_n\del_v$ 
 with
corresponding  screen vector field 
\[Z=\del_u- H\del_v -2g^{ij}\mu_i\del_j,\]
where $g^{ij}$ is the inverse matrix of $g(\del_i,\del_j)$.
Moreover, if   the $\mu_i$'s in \eqref{coord} are the coefficients of a $u$-dependent family of closed one-forms, this screen is also involutive.

The term {\em horizontal} for a given screen distribution comes from the following observation 
about the Riemannian metric defined in \eqref{h}. Its proof is a straightforward computation using the Koszul formula.

\bs \label{prop1}
Let $(\cal M,g)$ be a \edrevised{Brinkmann space}. By  $\cal N$ we denote  a 
leaf of the integrable distribution $V^\bot$. Furthermore,  let $\SS $ be a screen distribution defining a Riemannian metric $h$ on $\M$ and an induced Riemannian metric $h^\cal N$ on each leaf $\cal N$ by restriction.   Then we have
\bnum
\item 
$\SS $ is horizontal 
if and only if on each leaf $\cal N $ of $V^\bot$,   $V$ defines an isometric Riemannian flow, i.e., $V$ is a Killing vector field of constant length for the metric $h^{\cal N}$ on $\cal N$.
\item Fix a leaf $\cal N$ and assume that, along $\cal N$, there is a screen 
distribution $\SS$ which is involutive  and horizontal. Then $V\in \Gamma (T\cal N)$ is  parallel on 
$(\cal N, h)$, i.e., with respect to the Levi-Civita connection $\nabla^h$ of $h$. 
In particular, the leaves of $\SS$ in $\cal N$ are totally geodesic for $h$.
\enum
\es
Now we derive some criteria 
 in terms of  $Z^\flat$
for a screen distribution to be horizontal and involutive. An immediate consequence of $\SS=V^\bot\cap Z^\bot$ is:


\blem
\label{LemCharInvolAndHor}
Let $\SS$ be a screen distribution with screen vector field $Z$. Then
	$\SS$ is involutive and horizontal if and only if 
	$dZ^\b|_{V^\bot\wedge V^\bot}=0$, which is equivalent to $V^\b\wedge dZ^\b=0$.
\elem


Now we assume that the  the screen bundle $\Sigma$ is globally trivialisable, i.e., that  $\Sigma$ admits $n$ linearly 
independent global sections. Then, for each screen distribution $\SS$, the bundle projection
\[\SS\ni X\longmapsto [X]\in \Sigma\]
gives a global \edrevised{orthonormal} frame field $S_1, \ldots, S_n$ for $\SS$ from a trivialisation of $\Sigma$. 
\edrevised{
We denote by $S^i=S_i^\flat=\delta^{ij}g(S_j,.)$ the corresponding metric duals,
where $\delta^{ij}$ denotes the Kronecker-symbol and we use the Einstein summation convention.}
Then we define 
global one-forms
\begin{equation}
	\label{alpha} \alpha^i := \nabla S^i(Z)=\delta^{ij}g(\nabla  S_j,Z),
\end{equation}
and 
	a direct computation shows
that $\SS $ is horizontal if and only if \begin{equation}
	\label{alpha-horiz}\alpha^i(V)=0,
\end{equation}
and $\SS$ is involutive if and only if 
\begin{equation}
	\label{alpha-integ} \alpha^i(S_j)-\alpha^j(S_i)=0.
\end{equation}
Now we compute the difference between two screen vector fields and their screen distributions,
still under the assumption that the screen bundle $\Sigma$ is globally trivial, so that we have a global orthonormal 
frame field $\s_1, \ldots , \s_n$ of $\Sigma$. Then, if $\SS$ and $\hat\SS$ are two screen distributions, 
the sections $\s_i$ define sections $S_i\in \Gamma(\SS)$ and $\hat S_i\in \Gamma(\hat\SS)$, both 
orthonormal with respect to $g$,  which are related by
\[\hat{S}_i=S_i-b_i V \longmapsto \s_i=[S_i]\in \Gamma(\Sigma),\]
for smooth functions $b^i$ on $\cal M$. 
\edrevised{
For the dual relation we have
\[
\hat S^i=S^i-b^i V^\flat,
\]
with functions $b_i=b^i$.}
The corresponding screen vector fields 
$Z$ and $\hat Z$ are then related by
\edrevised{
\[\hat Z=Z+ b^kS_k-\frac{1}{2}\, b^kb_k\, V,\]}
and for the differentials of the duals we get
\edrevised{\[
d\hat{Z}^\flat=dZ^\flat
+d b_k\wedge S^k +
b_k dS^k
- b_k db^k\wedge V^\flat.\]}
Then, computing the differentials of $Z^\b$ and $d\hat S^\b_i$ we get
\[
dZ^\b= \delta_{kl}S^k \wedge \alpha^l
\]
with $\alpha^k$ defined in \eqref{alpha}, 
and
\edrevised{\[dS^i=\w^i_{\;k}\wedge S^k+\alpha^i\wedge V^\b,\]}
where $\w^i_{\;k} $ is the part of the  connection one-form defined by 
\edrevised{\belabel{omegas} \omega^i_{\;k}=- \nabla  S^i(S_k).\end{equation}}
This allows us to express  the differential of $\hat Z^\b$ in terms of a basis of the old screen, its connection coefficients and the functions $b^i$ as
\edrevised{
\[
d\hat Z^\b= (db_k +b_l\w^l_{\; k} -\delta_{kl}\alpha^l)\wedge S^k + b_k(\alpha^k-db^k)\wedge V^\b,
\]}
This, together with Lemma~\ref{LemCharInvolAndHor} gives us

\bs \label{propinvandhor}
Let $(\M, g)$ be a \edrevised{Brinkmann space} and a screen bundle that is defined by global sections $S_1, \ldots, S_n$, and let $\alpha^i$ and $\omega^i_{~j}$ be the corresponding connection forms defined in \eqref{alpha} and \eqref{omegas}. Then there is
 an involutive and horizontal screen 
distribution if and only if there are smooth functions $b_1, \ldots , b_n$ on $\M$ which are solutions to the differential system 
\edrevised{
\belabel{diffs1system}
(db_k +b_l\w^l_{\; k} -\delta_{kl}\alpha^l)\wedge S^k|_{V^\bot\wedge V^\bot}.
\end{equation}
In particular, if there exist functions $b_i$ such that  
\belabel{diffs1} (db_k +b_l\w^l_{\; k} -\delta_{kl}\alpha^l)|_{V^\bot}=0,\end{equation}}
then there is a horizontal and involutive screen distribution spanned by $S_i - b_iV$.
\es

\subsection{Manifolds with trivial screen holonomy and pp-waves}

We say that a \edrevised{Brinkmann space} $(\cal M,g)$ 
has {\em trivial screen holonomy} if the full holonomy group of $\nabla^\Sigma$ is trivial, 
i.e., consists only of the identity transformation. This is related to the  notion of 
pp-waves as in Definition \ref{ppdef}.
\bs\label{pp-prop}
 Let  $(\cal M,g)$ be a \edrevised{Brinkmann space} with parallel null vector field $V$.
 \bnum
 \item The following statements are equivalent:
 \bnum
 \item 
  $(\cal M,g)$  is 
 a {\em pp-wave}.
 \item \label{screen-flat2}
For all $W \in V^\bot$ and $ X,Y\in \Gamma(T \cal M)$ it holds
$
R(X,Y)W \in \rr V
$.
 \item The screen bundle $(\Sigma, \nabla^\Sigma)$ is flat, 
i.e.,  the curvature of $\nabla^\Sigma$ vanishes.
\item The connected component of the holonomy group of $(\cal M,g)$ is contained in $\rr^n\sub \mathrm{SO}(1,n+1)$.
\item There exist local sections $S_1, \ldots S_n$ of $V^\bot$ with $g(S_i,S_j)=\delta_{ij}$ and  local one-forms $\alpha^i$ such that 
\edrevised{
$\nabla S_i= \delta_{ij}\alpha^i\otimes V$ as in \eqref{alpha}}.
In this case, the one-forms satisfy $d\alpha^i |_{V^\bot\wedge  V^\bot}=0$.
\enum
 \item The holonomy of $\nabla^\Sigma $ is trivial if and only if the holonomy of $(\cal M,g)$ is contained in $\rrn$. 
 \enum
\es
\bprf
\edrevised{
The proof of (1a) $\Leftrightarrow$ (1b) follows directly from the definition of a pp-wave. Moreover, (1b) $\Leftrightarrow$ (1c)
is straightforward by the definition of $\nabla^\Sigma$. The computations proving the equivalence of (1a) and (1d) are carried out in \cite[Theorem 4.2]{leistner01}, while (2) is obvious by 
\cite[Proposition 2(4)]{baum-laerz-leistner12}.
Finally, the existence of sections $S_i$ as in property (1e) is equivalent to (1c) since, locally, the existence of a $g^\Sigma$-orthonormal basis of $\nabla^\Sigma$-parallel sections $\sigma_i \in \Gamma(\Sigma)$ which
constitute a frame $S_i$ as in (1e) is equivalent to the flatness of $(\Sigma, \nabla^\Sigma)$. The property for the differentials of the $\alpha^i$'s, however, follows from the following computation: Let $Z$ be a screen vector field, $X\in V^\bot$ and $S_i$  frame fields as in (1e). Then
\[
d\alpha^i(S_j,X)
=
g(	R(S_j,X)S_i ,Z)
=
g(	R(S_i ,Z)S_j,X)
=
d\alpha^j(S_i,Z)g(V,X)
=
0,\]
for all $i,j=1, \ldots , n$, i.e., $d\alpha^i |_{V^\bot\wedge  V^\bot}=0$.
}
\eprf


Clearly, manifolds with trivial screen holonomy are pp-waves, but for non simply connected manifolds the converse is not 
true  (see \cite{baum-laerz-leistner12} for examples).
%

Locally, for a pp-wave the coordinates in \eqref{coord} can be chosen in a way such that 
$\mu_i\equiv 0$ and $\hat{g}_{ij}\equiv \delta_{ij}$, i.e., with $\hat g$ being the 
standard flat metric for all $u$, i.e.,
\[
g=2du (dv + Hdu )+  \delta_{ij}dx^idx^j
\]
where $H = H(u,x^1, \ldots ,x^n)$ is a smooth function.
\edrevised{This can be seen by computing the curvature for a
metric of the form  \eqref{coord}. 
Evaluating the curvature conditions defining a pp-wave will then imply that the metric $\hat{g}_{ij}$ is flat and hence can be chosen as $\hat{g}_{ij}=\delta_{ij}$, and that the $1$-form $\mu_i$ is closed, which allows to change coordinates in a way that $\mu_i=0$ (for an explicit proof see \cite{schimming74} or the appendix of \cite{GlobkeLeistner14}).}
 In these coordinates, $\nabla \del_v=0$ and
\[\nabla_{\del_i} \del_j=0,\ \nabla_{\del_i} \del_u=\del_i(H)\del_v,\ \nabla_{\del_u} \del_u=\del_u(H)\del_v -\sumi \del_i(H)\del_i\]
which implies that the only non-vanishing curvature terms of $g$, up to symmetries, are
\begin{equation}\label{curv}
R (\del_i,\del_u,\del_u,\del_j)=-\del_i\del_j H,
\end{equation}
where our sign convention is $R (X,Y)U=\left[ \nabla _X,\nabla _Y\right]U-\nabla _{[X,Y]}U$, 
and for the Ricci curvature, $Ric=\mathrm{trace}_{(2,3)}R $,
\[Ric(\del_u,\del_u)=\Delta (H)
\]
where $\Delta=-\sumi \del_i^2$ is the flat Laplacian. Formula \eqref{curv} shows that the connected  holonomy of a pp-waves is {\em equal} to $\rrn$, and hence indecomposable, if there is a point in $\M$ with local coordinates such that the Hessian of $H$ is non-degenerate at this point.

Since the distribution $V^\bot$ is parallel and thus defines a foliation of $\cal M$ into totally 
geodesic leaves of codimension one, the flatness of the screen bundle can be stated as


\blem
Let $(\cal M, g)$ be a Lorentzian manifold with parallel null vector field $V$ and foliation $V^\bot$. Then $(\cal M, g)$ is a pp-wave if and only if, for each leaf
$\cal N$ of $V^\perp$,
 the linear 
connection which is induced on $\cal N$ by the Levi-Civita connection of $g$ is flat.
\elem


	\bprf
	Let $\nabla^{\cal N}$ be the linear connection defined by $\nabla $ on a leaf $\cal N$ of 
	$V^\bot$, i.e., $\nabla^{\cal N}_UW:=\nabla _UW\in V^\bot|_{\cal N}$ for 
	$U,W\in \Gamma(T{\cal N})$, where $T{\cal N} = V^\bot|_{\cal N}$. 
	Hence, for the curvature $R$ of $\nabla$ and $R^{\cal N}$ of $\nabla^{\cal N}$ we have
	\[R^{\cal N}(U,W)S=R (U,W)S \ \ \ \forall \ U,W,S\in V^\bot|_{\cal N}.\]
	This term vanishes if and only if $g(R (U,W)S,X)=0$ for all $X\in \Gamma(T\cal M)$, which is 
	equivalent to \eqref{screen-flat} in the definition of pp-waves on page \pageref{screen-flat}.
	\eprf 


\bs
\label{PropTrivHol}
Let $(\cal M,g)$ be a Lorentzian manifold with parallel null vector field $V$ and trivial screen 
holonomy, or equivalently, with \edrevised{$\Hol(\cal M,g) \subseteq \rrn$}. Then, for each screen distribution 
$\SS=V^\bot \cap Z^\bot$ with screen vector field $Z$, there is a global \edrevised{orthonormal} 
frame field $S_1, \ldots , S_n$ of $\SS$ \edrevised{with 
	\begin{equation}\label{nablasi}
\nabla S_i = \delta_{ij}\alpha^j \otimes V,
\end{equation} where the $\alpha^i$ agree with \eqref{alpha}}
and \edremoved{the $\alpha^i$  defined in \eqref{alpha}} satisfy
\[d\alpha^i(X,Y)=R(X,Y,S_i,Z),\]
and hence
\[d\alpha^i|_{V^\bot\wedge V^\bot} =0.\]
Furthermore, a given screen distribution $\SS$ can be changed to an involutive and horizontal 
one, if there exist functions $b^1, \ldots , b^n$ on $\cal M$ such that
\[(db^i-\alpha^i)|_{V^\bot}=0.\]
\es


	\bprf
	Since $\Sigma$ is assumed to have trivial holonomy, we find global basis sections 
	$\s_i$ of $\Sigma$ such that $\nabla^\Sigma  \s_i=0$. Hence, for a given screen 
	distribution, the induced frame fields $S_i$ satisfy 
	$[\nabla_XS_i]=\nabla^\Sigma_X\sigma_i=0$ and thus
\[\nabla  S_i=g(\nabla  S_i, Z)V =\nabla S^i(Z)V=\alpha^i\otimes V,\]
	or equivalently, $\w^i_{\; j}=0$. As above, we have
	\be
	R (X,Y)S_i 
	& = & d\alpha^i(X,Y) \cdot V,	
	\ee
	since $V$ is parallel.
	Given functions $b^i$ with $(db^i-\alpha^i)|_{V^\bot}=0$, from $\w^i_{\; j}=0$ and 
	equation~\eqref{diffs1} in Proposition~\ref{propinvandhor} we see that $\hat S_k=S_k - b_kV$ with $b^k=b_k$ defines a horizontal and 
	integrable screen distribution.
	\eprf
For Lorentzian manifolds with trivial screen holonomy admitting a horizontal and integrable 
screen distribution, we can strengthen Proposition~\ref{prop1} in the following way:


\bs \label{prop2}
Let $(\cal M,g)$ be a Lorentzian manifold with parallel null vector field $V$ and with 
trivial screen holonomy. Let $\cal N$ be a leaf of the integrable distribution $V^\bot$. 
Assume that, along $\cal N$\!, there is a screen distribution $\SS$ which is involutive and 
horizontal. Then  the Riemannian metric  $h$ on $\cal N$ that is defined by $\SS$ by 
the relations \eqref{h} is flat and the frame $S_i$ in  Proposition~\ref{PropTrivHol} together with $V$ constitutes a $\nabla^h$-parallel frame field for $(\cal N,h)$.
\es


	\bprf
	Let $\SS$ be an involutive and horizontal screen distribution along a $V^\bot $-leaf 
	$\cal N$. By Proposition~\ref{PropTrivHol}, we have sections $S_1, \ldots , S_n$ of $\SS$ 
	with \edrevised{\eqref{nablasi}}, i.e., \[g(\nabla _XS_i,Y)=0\]
	for all $X,Y\in \Gamma(T\cal N)$. Writing out the Koszul formula for this term
	we get
	\be
	0&=&X(g(S_i,Y))+S_i(g(X,Y))-Y(g(S_i,X))
	\\
	&&{ }+g([X,S_i],Y)+g([Y,S_i],X)+g([Y,X],S_i)
	\ee
	This equation holds for all $X,Y\in \Gamma(T\cal N)$, but, since $\SS$ was assumed to be 
	horizontal and involutive, we have that the brackets $[X,S_i]$, $[Y,S_i]$ and $[X,Y]$ 
	are in $\SS$. Hence, when recalling the definition of $h$ in \eqref{h}, in the above 
	expression we can replace the metric $g$ by the Riemannian metric $h$ on $\cal N$, 
	which shows that
	\[h(\nabla^h_XS_i,Y)=0.\]
	Hence, the $S_i$ are parallel vector fields on $(\cal N, h)$. But we have already 
	seen in Proposition~\ref{prop1} that $V$ is also parallel for $h$. Hence, 
	we have a $h$-orthonormal frame of $\cal N$ which is parallel for $\nabla^h$ 
	yielding the flatness of $(\cal N,h)$.
	\eprf
\edremoved{
\bbem
Let $(\cal M,g)$ be a pp-wave and $\SS$ an involutive and horizontal screen distribution 
on $\cal M$ and denote by $h$ the Riemannian metric on $\cal M$ defined by $\SS$.
 In this situation, we have seen that,
on each  leaf $\cal N$ of $V^\bot$,  the connection induced by $\nabla $  coincides 
with the Levi-Civita connection of   the  Riemannian metric $h^{\cal N}$  which is 
{\em induced by $h$ on $\cal N$}, and in fact, both are flat. However, on $\cal M$ 
 the Levi-Civita connections  of $(\M,g)$ and $(\M,h)$ do not 
coincide,  not even along $\cal N$.
\ebem}

\section{The universal cover of a pp-wave}
\label{coversection}

\subsection{Review of completeness results for Lorentzian manifolds with parallel null vector field}\label{seccomplete}
Before we turn to the universal cover of pp-waves  and to the proof of the main theorems for compact manifolds, we want to recall results about the completeness of 
 Lorentzian manifolds with parallel null vector fields
in a more general setting.
To our knowledge, the strongest of such results   can be found in 
 \cite{candela-flores-sanchez03}.
\setcounter{theo}{0}
\btheo[{\cite[Theorem 3.2 and Corollary 3.4]{candela-flores-sanchez03}}]
\label{Lem-Sanchez}
	Let $(\cal S, h)$ be a connected Riemannian manifold of dimension $n$ and let $H\in C^\infty(\rr\times \cal S)$ be a smooth function. On the manifold
$\cal M:=\rr^2\times \cal S$ define the Lorentzian metric $g$ by
\belabel{pfw}
g|_{(u,v,x)}=2dudv+2H(u,x)du^2+h|_x,
\end{equation}
where $x\in \cal S$ and $(u,v)$ are the global coordinates on $\rr^2$.
\bnum[(i)]
\item
The Lorentzian manifold $(\cal M,g)$ is geodesically complete if and only if the Riemannian manifold $(\cal S, h)$ is complete and the solutions 
 $s \mapsto \gamma(s)$ of the ODE
	$$
		\frac{\nabla^h \dot \gamma}{ds}(s) = \grad^hH(s,\gamma(s))
	$$	
	are defined on the whole real line. Here $\nabla^h$ is the Levi-Civita connection of $h$.
\item If $(\cal S,h)$ is geodesically complete and  the function $H$
does not depend on $u$ and is at most quadratic at spacial infinity,
i.e., there exist $x_0\in \cal S$ and real constants $r,c >0$ such that
\[ H(x) \leq c \cdot d_{\cal S}(x_0,x)^2, \quad \mbox{ for all } x\in \cal S \mbox{ with } d_{\cal S}(x_0,x) \geq r,\]
then $(\cal M,g)$ is geodesically complete. Here $d_{\cal S}$ is the
distance function of $(\cal S, h)$.
\enum
\etheo
This theorem applies to pp-waves in standard form, and  to the more general class of pp-waves that are globally of the form \eqref{pfw} with $(\cal S,h)$ a flat Riemannian manifold, not necessarily the $\rr^n$.
 For plane waves (as defined in the introduction) it implies
 \bs[{\cite[Proposition 3.5]{candela-flores-sanchez03}}]
 Let $(\cal M,g)$ be a Lorentzian metric of the form \eqref{pfw} and assume that its universal cover is globally isometric to  a standard plane-wave
 \[\big(\rr^{n+2},
\wt g=2du (dv + (a_{ij}(u)x^ix^j) du )+  \delta_{ij}dx^idx^j
\big),\] with $a_{ij}=a_{ji}\in C^\infty(\rr)$.
 Then $(\cal M,g)$ is geodesically complete.
 \es
Regarding the Ehlers-Kundt problem, these results imply:
\bs{\cite[Theorem 4]{flores-sanchez06}}
 Any gravitational (Ricci-flat and four-dimensional) pp-wave (in standard form) such that $H$ behaves at most quadratically at spatial infinity (in the sense of Theorem~\ref{Lem-Sanchez}) is a (necessarily complete) plane wave.
\es

%

\subsection{Horizontal and involutive screen distributions on the universal cover}

In this section we will deal with compact pp-waves but first present some results which 
hold in a more general setting. A vector field is {\em complete} if its maximal integral curves are defined on $\rr$.


\bs
\label{PropUnivSplit}
Let $\cal M$ be a  manifold with a closed nowhere-vanishing one-form $\eta$. Assume that there is a complete 
vector field $Z$ with  $\eta(Z)=1$. Then the leaves of the distribution $\Ker(\eta)$ 
are all diffeomorphic to each other under the flow $\phi_t$ of $Z$, and the universal 
cover $\tem$ of $\cal M$ is diffeomorphic to $\rr\times \widetilde{\cal N}$
with the diffeomorphism  given as
\begin{equation}\label{flow}
\Phi :\rr\times \ten \ni(u,p) \mapsto \wt\phi_u(p)\in \tem,
\end{equation}
where $\ten$ is the universal cover of a leaf of the distribution $\Ker(\eta)$ and $\wt \phi$ is the flow of the lift of $Z$.
If $\cal M$ is compact with closed $\eta$, all of the above 
is satisfied.
\es


	\bprf 
	The idea of the proof can be found in \cite[Thm. 3.1]{milnor63}.
	Since $\eta $ is closed, the distribution $\Ker (\eta)$ is involutive and
	 the Lie derivative of $g$ is
	\[\cal L_Z\eta (X)=d\eta(Z,X)+X(\eta(Z))\equiv 0.\] 
	For each $t\in \rr$, the flow $\phi_t$ of $Z$ 
	is a diffeomorphism of $\cal M$. Then $\cal L_Z\eta=0$
	 shows that $\phi_t$ maps  the leaves of the distribution $\Ker(\eta)$ 
	diffeomorphically onto each other.
	
	Let $\wt\eta $ and $\wt Z$ be the lifts of $\eta$ and $Z$ to the universal cover 
	$\tem$. Then $\wt{Z}$ is still a complete vector field with $\wt\eta(\wt Z)=1$ and 
	$d\wt\eta=0$. Hence, there is a real function $f\in C^\infty(\cal M)$ such that 
	$\wt\eta =df$. Let $\wt\phi_t$, $t\in \rr$,  denote the flow of $\wt{Z}$. Then, 
	for each $p\in \tem$, the function $\tau_p:\rr\to \rr$ defined by 
	$\tau_p(t):= f(\wt \phi_t(p))\in \rr$ satisfies
	\[
	\tau'_p(t)=df_{\phi_t(p)}(\wt Z)=\eta_{\wt\phi_t(p)}(\wt Z) \equiv 1.
	\]
	Hence, $\tau(t)=t+f(p)$. This shows that $f:\tem \to \rr$ is surjective and that 
	two level sets $\ten_a=f^{-1}(a)$, $a\in \rr$, are diffeomorphic under the flow,
	\[
	\wt \phi_{b-a}:\ten_a\simeq \ten_b .\]
	This defines a diffeomorphism
	\be
	\Phi\ :\  \rr\times \ten_0&\lra &\tem
	\\
	(u,p)&\longmapsto&\wt\phi_u(p)
	\ee
	the inverse of which is given by
	\[\Phi^{-1}(p)=(f(p),\wt\phi_{-f(p)}(p))\in \rr\times \ten_0.\]
Being simply connected,  $\ten:=\ten_0$   is  the universal cover 
	of the leaves of $\Ker(\eta)$.
	\eprf

The proposition applies in particular to a compact Lorentzian manifold with parallel 
null vector field $V$ defining a closed one form $\eta=g(V,.)$. Here $\ten$ is  
an integral manifold of the distribution $\wt V^\bot$. 
\begin{theo}	
	\label{ThmHI1}
	Let $(\cal M, g)$ be an $(n+2)$-dimensional pp-wave with parallel null vector field 
	$V \in\Gamma (T\cal M)$ and complete screen vector field. 
	Then there exists a horizontal 	and involutive screen distribution $\SS$ on the 
	universal cover $(\tem, \widetilde{g})$ of $(\cal M, g)$. \edrevised{Moreover}, there are 
	linear independent \edrevised{$\widetilde{g}$-orthonormal} vector fields  $S_i\in \Gamma(\SS)$ on $\tem$, $i=1, \ldots , n$, 
	with $\tnab_XS_i=0$,
	for all $X\in \widetilde{V}^\bot$ and $\tnab$ the Levi-Civita connection of $\tg$. \edrevised{In particular},  $(V,S_1, \ldots , S_n)$ is a $\nabla^h$-parallel orthonormal frame on $(\wt{\cal N},h)$, where  $h$ is the Riemannian metric defined by $\SS$ on the   leaves of $\wt{V}^\perp$.	
\end{theo}


\bprf
	By a tilde we denote the lift of any object to the universal cover $\tem$ of $\cal M$.
 Let $\SS=V^\bot \cap Z^\bot$ be a screen distribution defined by a complete screen vector field $Z$.	By assumption, the bundle 
 $\widetilde{\Sigma} \longrightarrow \tem$ is flat, and, since $\tem$ is simply connected, has trivial holonomy.	 Thus we obtain linearly independent 
	global parallel sections $\sigma_1, \ldots, \sigma_n \in \Gamma(\widetilde{\Sigma})$ which 
	give rise to $n$ sections	$S_1, \ldots, S_n \in \Gamma(\widetilde{\SS})$ with 
	$\tnab  S_i = \alpha^i \otimes \widetilde{V}$ with $\alpha^i:=\tg(\tnab  S_i,\wt Z)$. 
	Since $\cal M$ is a pp-wave, according to Proposition~\ref{PropTrivHol}, they satisfy
	\begin{equation}
		\label{dalpha}
		d\alpha^i|_{\widetilde{V}^\bot \wedge \widetilde{V}^\bot} = 0.
	\end{equation}	
	By Proposition~\ref{PropUnivSplit}, the universal cover $\tem $ is diffeomorphic to 
	$\R \times \cal{N}$, where $\cal {N}$ is 
	the universal cover of the leaves of the distribution $V^\bot$, and the map $\R\times \cal N\to \tem$ 
	is given by the flow of $\wt Z$. Now, for each $r\in \rr$, let 
	\be
	\iota_{(r)}: \cal N &\hookrightarrow & \R \times \cal  N \\
	x&\mapsto &(r,x)
	\ee
	denote the inclusion of $\cal N$ into $\R\times \cal N$. We use these to pull back the 
	$\alpha^i$'s to $\cal N$, 
	\[\alpha_{(r)}^i:=(\iota_{(r)})^*\alpha^i,\]
	which is now a one-parameter family of one-forms on $\cal N$, depending smoothly on 
	the parameter $r\in \R$. Because of equation \eqref{dalpha}, all $\alpha^i_{(r)}$ 
	are closed,
	\[d \alpha_{(r)}^i=d( \iota_{(r)}^*\alpha^i)=\iota_{(r)}^*d\alpha^i =0 .\]
	Fixing  $x_0\in \cal N$, since $\cal N$ is simply connected, for each $i=1, \ldots , n$ 
	and each $r\in \R$ we find a unique  function $b^i_{(r)}\in C^\infty(\cal N)$ such that
	\[ db^i_{(r)}=\alpha_{(r)}^i,\ \text{ and }\ b^i_{(r)}(x_0)=0,\]
	where the differential is the differential on $\cal N$. Hence we obtain smooth functions 
	$b^i\in C^\infty (\R\times \cal N)$ defined by
	\[b^i(r,x)=b_{(r)}^i(x).\]
	We have to verify that these functions are indeed smooth on $\R\times\cal  N$: Take an arbitrary
	$\hat x\in \cal N$ and fix coordinates \edrevised{$(\cal U,\vf=(x^1, \ldots , x^{n+1}))$ }
	around $\hat x$ 
	such that \edrevised{$\vf(\cal U)$ is star-shaped} and $\vf(\hat x)=0$. Over $\cal U$ we write $\alpha_{(r)}^i$ as 
	\[\alpha_{(r)}^i|_{(r,\vf^{-1}(x^1,\ldots , x^{n+1}))} = \sum_{k=1}^{n+1} \alpha^i_k(r,x^1, \ldots , x^{n+1})dx^k\]
	with $\alpha_k^i$ smooth functions on $\R \times \vf(\cal U)$ and the solutions $b_{(r)}^i$ 
	are given by
	\be
	b^i(r,\vf^{-1}(x^1, \ldots ,x^{n+1}) ) & = & b_{(r)}^i\circ \vf^{-1}(x^1, \ldots ,x^{n+1}) 
	\\
	& = & \sum_{k=1}^{n+1} x^k\!\!\!\int_0^1 \alpha_k^i(r,  t x^1, \ldots , tx^{n+1})dt 
	+ b^i(r,\hat x), \ee
	for all $(x^1,\ldots , x^{n+1})\in \vf( \cal U)$. Since $\alpha_{(r)}^i$ and hence $\alpha_k^i$ 
	depend smoothly on $r$, this is  smooth in $r$ and $x^i$, \edrevised{ as soon as $b^i(,\hat x)$ is smooth in $r$. Now choosing $\hat x=x_0$ first, because of $b^i(.,x_0)\equiv 0$, we see that $b^i$ is smooth on $\rr\times \cal U$, where $\cal U$ is a star-shaped neighbourhood of $\hat x=x_0$. Then covering $\cal N$ by star-shaped neighbourhoods, this argument shows that $b^i$ is smooth on $\rr\times \cal N$.}
	Using these $b^i\in C^\infty(\R\times \cal N)$ we define the new screen distribution  
\[		\widehat{\SS} := \operatorname{span}\{\hat S_1, \ldots, \hat S_n\}
		\ \ \text{ with } \ \
		\hat S_i := S_i - b^i  \wt V.
\]
	For every $\hat X=d \iota_{(r)}|_x (X)\in V^\bot_{(r,x)}\subset T_{(r,x)}(\R\times \cal N)$ 
	with $X\in T_x\cal N$ the 
	 $\hat{S}_i$  satisfy
		\[
		\tnab _{\hat X} \hat S_i\  =\ \bigl( \alpha^i (\hat X) - db^i(\hat X)\bigr) \wt V \ =\ 
	\left(	 \alpha^i_{(r)} (X) - db^i_{(r)}(X)\right) \wt V
		 \ = \ 0,\] 
	 which shows that $\hat\SS$ is involutive and horizontal.
	 Finally, the $\nabla^h$-parallelity of the frame $(V,S_1,\ldots, S_n)$ on $(\wt{\cal N},h)$ follows from Proposition~\ref{prop2}.
	\eprf


\bbem
	The  horizontal and involutive screen distribution on the universal cover obtained by this result does not 
	necessarily descend to a  horizontal and involutive one on
	the base manifold. In fact, the next example exhibits a \textit{compact} pp-wave for which 
	\textit{no  involutive realization of the screen bundle $\Sigma$ exists} (see also
 \cite[Proposition 2.42]{laerz-diss}).
\ebem

\bbsp\label{bundleexample}
We will give an example of a compact pp-wave that does not admit an involutive \edrevised{\textit{and} horizontal} screen distribution.  Let $\cal N:=\mathbb{T}^{n+1}$ be the $(n+1)$-torus  and
 let $c\in H^2(\T^{n+1},\Z)$ a  non-zero cohomology class, $\omega \in c$ a closed two-form representing $c$ in the de Rham cohomology. Now let $\pi:\cal M\to \T^{n+1}$ be the circle-bundle over $\T^{n+1}$ with first Chern class being equal to $c$. Furthermore, let $A\in T^*\M\otimes \mathrm{i}\rr$ be the corresponding $\mathrm{S}^1$-connection with curvature $F:=dA=-2\pi \mathrm{i}\pi^*\omega$.
Now let $\del_0, \ldots, \del_n$ be the canonical frame and $\xi^0, \ldots , \xi^n$ be the canonical coframe on $\T^{n+1}=\mathrm{S}^1\times \ldots \times \mathrm{S}^1$. Denote by
\be \eta:=\pi^*\xi^0,&& \s^i:=\pi^*\xi^i,\ i=1, \ldots n
\ee
 their pull-backs to $\tem$ and by $U$, $S_i$, $i=1, \ldots , n$ the $A$-horizontal lifts of $\del_0$ and $\del_i$, $i=1, \ldots, n$ respectively. Hence, we have
\be
d\pi_p(U)=\del_0|_{\pi(p)},&&
d\pi_p(S_i)=\del_i|_{\pi(p)}
\ee
and 
\[ A(U)=A(S_i)=0,\ \s^i(S_j)=\delta^i_{~j},\ \s^i(U)=0,\ \eta(S_i)=0,\ \eta(U)=1.\]
Let $V$ be the
fundamental vector field of the $\mathrm{S}^1$-action on $\cal M$, i.e., with $A(V)=\mathrm i$. Since $U$ and the $S_i$'s are defined as horizontal lifts of the $\del_i$'s we have that $[U,S_i]$ and $[S_i,S_j]$ are vertical vector fields and moreover that
\belabel{Vhoriz} [V,U]=[V,S_i]=0.\end{equation}

Having $(n+2)$-linearly independent nowhere-vanishing one forms, and choosing a smooth function   $H\in C^\infty (\mathbb{T}^{n+1})$, enables us to define a Lorentzian metric $g$  on $\cal M$ by
\[ g= 2 \left(H\eta -\mathrm{i} A\right)\cdot \eta +\sumi (\s^i)^2.
\]
The Koszul formula, the verticality of 
$[U,S_i]$ and $[S_i,S_j]$ together with equation \eqref{Vhoriz} show that $V$ is a parallel vector field for $g$. 
 An obvious screen vector field is given by
\[Z:=U-HV,\]
with the metric dual given by
\[Z^\b=g(Z,.)= H\eta-\mathrm{i}A.\]
Moreover, we have $\nabla_VS_i=0$, and, again using the Koszul formula
\beqa
\nabla_{S_i}S_j&=&\frac{\mathrm{i}}{2}F(S_i,S_j) V\label{nabsisj}
\\
\nabla_{U}S_i&=&\left( S_i(H)- \mathrm{i}F(S_i,U)\right)V +\frac{\mathrm{i}}{2}\sumk  F(S_i,S_k) S_k,\label{nabusi}
\eeqa
where $F=dA$ is the curvature of $A$.  For the curvature $R$ of $\nabla$ this implies
\[
R(S_k,U,S_i,S_j)
=
\frac{\mathrm{i}}{2} S_k( F(S_i,S_j))
=
\frac{\mathrm{i}}{2} (\nabla_{S_k}F) (S_i,S_j).
\]
Hence, in order to make $g$ a pp-wave metric, we have to assume that \belabel{nabF}
\nabla F|_{V^\bot \otimes V^\bot \wedge V^\bot}=0.\end{equation}
Then, for   the curvature of $g$ we obtain
\[
R(S_i,U,S_j,U)=
\del_i\del_j(H)-\frac{\mathrm{i}}{2}\Big(
S_i\left( F(S_j,U)\right) +\sumk F(S_i,S_k)F(S_k,S_j)\Big).
\]
Turning to the  properties of the screen, $\nabla_VS_i=0$ 
 implies that  the screen distribution $\SS$ defined by $Z$ is  horizontal. In order to show that $\SS$ is not involutive, we obtain from equation \eqref{nabsisj} that
\[Z^\b([S_i,S_j])= \mathrm{i} F(S_i,S_j).
\]
Hence, the screen defined by $Z$ is involutive, if and only if 
\[F|_{V^\bot\wedge V^\bot}=0,\]
or equivalently $0=\eta\wedge F$, since  $V^\b=g(V,.)=\eta$.
As an 
 example with condition \eqref{nabF}, but with $\eta \wedge F\not=0$, 
 in $F=-2\pi \mathrm{i}\pi^*\omega$ we can choose $\w\in \Omega^2(\T^{n+1})$ as 
\belabel{omegadef}
\w = \tfrac{1}{2}\sum_{i,j=1}^n a_{ij}\xi^i\wedge \xi^j \neq 0
\end{equation}
with constants $a_{ij} = -a_{ji}$. Moreover, since the $\xi^i$'s are not exact,  we can choose these constants such that $[\w]\not= 0\in H^2(\T^{n+2},\Z)$.  For such an example, the screen defined by $Z$ is not involutive, since 
\[\mathrm{i}F= \sum_{i,j=1}^n a_{ij}\s^i\wedge \s^j\]
does not vanish on $V^\bot\wedge V^\bot$. Then, $\cal M = \mathrm{S}^1 \times \cal P$, where
$\pi : \cal P \to \mathbb{T}^n$ is the circle bundle with first Chern class $c = [\omega] \in H^2(\mathbb{T}^n, \Z)$.

Now assume that $\hat \SS=\mathrm{span} (\hat S_1, \ldots, \hat S_n)$ is any other screen distribution defined by smooth functions
$b_i$ on $\M$  via $\hat S_i=S_i-b_i V$. This screen is horizontal since
\[0=\nabla_V\hat S_i=-V(b_i) V,\]
and the functions $b_i$ descend to smooth functions on $\T^{n+1}$. Furthermore, if  $\hat \SS$ was  involutive, we would have that
\[
0
=
\mathrm{i}F(S_i,S_j)-(d b_j(S_i)-db_i(S_j)),
\]
on $\M$,
or equivalently, for the one-form $\beta =\sumi \varphi_i \xi^i$  on $\T^{n}$,
\[\w (\del_i,\del_j)=d\beta (\del_i,\del_j),\]
 where $\varphi_i(x) := b_i(1, y)$ for $y \in \pi^{-1}(x)$. Hence,
$\w = d\beta$
which contradicts $[\w]\not=0$.
\ebsp

 
\subsection{The universal cover of  pp-waves under completeness assumptions}
Now we will  use the results in the previous section to show that, under some completeness  conditions, the universal cover of a pp-waves is globally of standard form.

\btheo\label{covertheo}
Let $(\M,g)$ be a pp-wave with parallel null vector field $V$ satisfying the following completeness  assumptions:
\bnum[(i)]
\item\label{theo3ass1}
 The maximal
	geodesics along the leaves of the parallel distribution $V^\bot$  are defined on $\rr$, and
	\item there exists a complete screen vector field $Z$.
	\enum
	Then the universal cover $\tem$ of $\M$ is diffeomorphic to $\rr^{n+2}$. Moreover, the universal cover $(\tem,\tg)$ is globally isometric to a
	standard  pp-wave	\begin{equation}	\label{equ-MT2a}
			\left(  \R^{n+2}, \ \ g^H = 
			2dudv + 2H(u,x^1, \ldots, x^n)du^2 + \delta_{ij}dx^idx^j\right).
	\end{equation}
	Under this isometry, the lift $\wt V$ of the parallel vector field $V$ is mapped to $\del_v$.
	\etheo
	
	\bprf
	By a tilde we shall denote the lift of an object to the 
	universal cover $\tem$ of $\cal M$.
					First, let $Z$ be the complete screen vector field and $\SS$ the corresponding screen 
	distribution on $\cal M$. Since $Z$ is complete,  we can apply 
	Proposition~\ref{PropUnivSplit} to $Z$ and  $\eta:=g(V,.)$ and obtain that the universal 
	cover $\tem$ of $\cal M$  is diffeomorphic to $\R \times \widetilde{\cal N}$, where 
	$\widetilde{\cal N}$ is the universal cover of a leaf $\cal N$ of the distribution $V^\bot$ 
	of $\cal M$. Clearly, $\widetilde{\cal N}$ is also a leaf of the distribution 
	$\widetilde V^\bot $ on $\tem$.
	\edremoved{
	Again, as  $Z$ is complete we can apply Theorem~\ref{ThmHI1} and obtain a horizontal 
	and involutive realization $\widehat{\SS}$ of the screen bundle $\widetilde{\Sigma}$ on 
	the universal cover $\tem$ and a corresponding screen vector field 
	$\hat Z\in \Gamma(T\tem)$, together with a frame $\hat S_i$ of $\widehat{\SS}$ with \belabel{nabsi0}
	\nabla_X\hat S_i=0, \text{ for all $X\in \wt V^\perp$.}
	\end{equation}
	 Furthermore,
	consider  the Riemannian metric $\hat h$ on $\widetilde{\cal N}$ defined by $\widehat{\SS}$,
	\[
	\hat h(\widetilde V,\widetilde V)=1,\ \hat h|_{\hat\SS\times \hat\SS} 
		= g|_{\hat\SS\times \hat\SS}, \  \hat h(\widetilde V,.)|_{\hat\SS} = 0.
	\]
	Hence, by Proposition~\ref{prop2},  the  simply connected Riemannian manifold $(\wt{\cal N}, \hat h)$ is flat and  admits a 
	frame $(\wt V, \hat{S}_1, \ldots, \hat{S}_n)$ of $\nabla^{\hat h}$-parallel vector fields which are 
	orthonormal with respect to $\hat h$. Moreover, by the first completeness assumption and relation~\eqref{nabsi0}, they are complete. Using  a result by Palais 
	\cite[Theorem VIII, Chapter IV]{palais57} (see also \cite[Proposition 1.9]{tricerri-vanhecke83} for a proof), we conclude that
	the manifold $\wt{\cal N}$ has a unique  structure of an Abelian Lie group for which the frame and 
	its (in our case vanishing) Lie brackets define the Lie algebra, and therefore, being simply connected,
	${\wt{\cal N}}$ is diffeomorphic to $\rr^{n+1}$.}
	\edrevised{Using assumption \eqref{theo3ass1} and because $(\cal M, g)$ is a pp-wave we conclude that $\wt\nabla|_{\wt {\cal N}}$ is a complete and flat connection on
	the simply-connected manifold $\wt{\cal N}$. Hence the exponential map w.r.t.\ $\wt\nabla|_{\wt {\cal N}}$ is a diffeomorphism and hence $\wt {\cal N}$ diffeomorphic to $\R^{n + 1}$.}
	This proves the first part of the statement.	
	
	The proof of the second part, that the lifted metric $\tg$ is isometric to a standard pp-wave, is more involved and requires some auxiliary statements. It follows  ideas in~\cite{derdzinski-roter09}.
	We start with
\begin{lem} 
	\label{Lemma:Jacobifields on flat manifolds}
	Let $(\cal M, \nabla)$ be smooth manifold with a torsion free connection 
 $\nabla$ and let $\delta : I \lra \M$ be a curve in $\cal M$ with 
	$0 \in I \subset \R$. Then a	vector field $X \in \Gamma(\delta^*T \cal M)$ along 
	the curve $\delta$ is parallel along $\delta$ if and only if  the vector field 
	$Y \in \Gamma(\delta^*T \cal M)$ with $Y(t) := t \cdot X(t)$ satisfies $\frac{\nabla^2}{dt^2}Y(t)\equiv 0$.
\end{lem}


	\begin{proof} One direction of the proof is trivial, so let us assume that
		\begin{equation}
			\label{Equation:Proof Jacobifields on flat manifolds no. 1}
			\frac{\nabla^2}{dt^2}Y(t)\equiv 0.
		\end{equation}
		By the Leibniz rule  this implies 
		that
		\begin{equation}
			\label{Equation:Proof Jacobifields on flat manifolds no. 2}
			2 \cdot \frac{\nabla} {dt} X(t) + t \cdot \frac{\nabla^2} 
			{dt^2} X(t) = 0
		\end{equation}
		for $t \in I$. Now let 
		$E_i(t) := \mathcal P_{\delta (t)}^\nabla(e_i)$ with fixed \edrevised{$x_0=\delta(0) \in \cal M$ }
		and a basis $e_1, \ldots, e_n$ in $T_{x_0}\cal M$. Consequently, we can write
		\edrevised{
		\begin{equation}
			\label{Equation:Proof Jacobifields on flat manifolds no. 3}
			X(t) = \sum_{i = 1}^n \xi^i(t) \cdot E_i(t),
		\end{equation}}
		which implies that $X$ is parallel along $\delta$ if $(\xi^i)' \equiv 0$ on $I$ 
		for all $i = 1, \ldots, n$. If we write $X$ in the form 
		(\ref{Equation:Proof Jacobifields on flat manifolds no. 3}),
		formula (\ref{Equation:Proof Jacobifields on flat manifolds no. 2}) 
		implies that the coefficient functions $\xi^i \in C^\infty(I)$ of $X$ must satisfy 
		the ordinary differential 
		equation $2(\xi^i)'(t) + t \cdot (\xi^i)''(t) = 0$ with the  initial values given 
		by $X(0) \in T_{x_0}\cal M$. Each such equation only has the constant solution defined on $I$:
		A discussion of the solutions $y$ on $(0, \infty) \cap I$ or $(-\infty, 0) \cap I$ 
		of $2y(t) + ty'(t) = 0$ yields \edrevised{$|y(t)| = \frac{C^2}{t^2}$ for some constant $C \in \R$}. Therefore, $y \equiv 0$ is the only
		solution, defined on $0$ since otherwise it must be equal to 
		\edrevised{$\pm \frac{C^2}{t^2}$} on $(0, \infty) \cap I$ or $(-\infty, 0) \cap I$ - a contradiction. 
		Hence, if $\xi^i \in C^\infty(I)$ with $0 \in I$ solves $2(\xi^i)'(t) + t \cdot (\xi^i)''(t) = 0$, 
		then $y := (\xi^i)'$ is defined on $0 \in I$ and solves the second differential equation of order one.
		Consequently, $y = (\xi^i)'$ is identically zero.
	\end{proof}
	To prove the second part of  Theorem \ref{covertheo}, let $\wt Z$ and $\wt\SS$ denote the lifts to $\wt{\cal M}$.
	Let $\gamma : \R \lra \wt{\cal M}$ be the integral curve of the complete vector field $\wt Z$
	through $x_0 \in \wt{\cal M}$ and $S_1, \ldots, S_n \in \Gamma(\wt\SS)$ such that $\wt g(S_i,S_j) = \delta_{ij}$ and 
	\edrevised{$\tnab S_i = \delta_{ij}\alpha^j \otimes \wt V$}, see Proposition~\ref{PropTrivHol}.
	We define the smooth map $\Phi : \R \times \R \times \R^n \lra \wt{\cal M}$ by
	\begin{equation}
		\label{Isom}
		\Phi(u,v,x^1,\ldots, x^n) := \exp_{\gamma(u)} (v \cdot \wt V(\gamma(u)) + \bigl\langle x, \vec{S}(u)\bigr\rangle ),
	\end{equation}
	where $\exp$ is the exponential map of $\tg$ and where we define $\bigl\langle x, \vec{S}(u)\bigr\rangle := \sumk x^k  S_k(\gamma(u))$. Then we  show
	\blem
	The smooth map $\Phi$  in \eqref{Isom} is well defined and a diffeomorphism.
	\elem
\bprf
By the first completeness assumption, the exponential	$\exp_p : \wt V^\bot_p = T_p\wt{\cal N} \lra \wt{\cal N}$ 
	is defined on the whole tangent space for each leaf $\wt{\cal N}$ through $p \in \wt{\cal M}$ 	
	and moreover, it is a diffeomorphism, since	$(\wt{\cal N}, \tnab |_{\wt{\cal N}})$ is a 
	complete, flat and simply connected manifold.
	Hence, in order to prove that $\Phi$ is injective, 
	it suffices to show that
	$\Phi(u_1,v_1,x) \neq \Phi(u_2,v_2,y)$ 
	for all $v_1,v_2 \in \R$ and $x,y \in \R^n$ whenever
	 $u_1 \neq u_2$.
But for  $u_1 \neq u_2$ we have
	$$
		\text{$\gamma(u_1)$ and $\gamma(u_2)$ are contained in two 
		disjoint leaves $\wt{\cal N}_1$ and $\wt{\cal N}_2$,}
	$$
	respectively. To see this, recall that --- as we have seen in the proof of 
	Proposition~\ref{PropUnivSplit} --- it holds $\wt \eta = df$ for some 
	$f \in C^\infty(\wt{\cal M})$ and $\wt\eta := \wt g(\wt V, \cdot)$ such that 
	$f(\gamma(u)) = u + f(x^0)$. In this situation,	each leaf is given as a level set of 
	$f$ and $\wt{\cal N}_1 = f^{-1}(u_1 + f(x^0))$, $\wt{\cal N}_2 = f^{-1}(u_2 + f(x^0))$ 
	which implies that $\gamma(u_1)$ and $\gamma(u_2)$ cannot lie within the same leaf. 
	But then 
	$$
		\exp_{\gamma(u_1)}(\wt V_{\gamma(u_1)}^\bot) = \wt{\cal N}_1 \ \ 
		\text{ and } \ \ \exp_{\gamma(u_2)}(\wt V_{\gamma(u_2)}^\bot) = \wt{\cal N}_2
	$$
	and since $\wt{\cal N}_1 \cap \wt{\cal N}_2 = \emptyset$, this yields 
	$\Phi(u_1,v_1,x) \neq \Phi(u_2,v_2,y)$ for all $v_1,v_2 \in \R$ and $x,y \in \R^n$.
	
	For proving the surjectivity of $\Phi$, let $p \in \wt{\cal M}$ be arbitrary and $\wt{\cal N}_p$ be the 
	leaf through~$p$. Then $f|_{\wt{\cal N}_p} \equiv c$ for some $c \in \R$.
	The point $(c_0,v,x)$ with $c_0 := c - f(x^0)$ and 
	$\exp_{\gamma(c_0)} (v  \wt V(\gamma(c_0)) + \bigl\langle x, \vec{S}(c_0)\bigr\rangle ) = p$ 
	is then a preimage of $p$.
	\eprf
	
Now  	we will show that the pull-back of $\tg$ by $\Phi$ is of the form $g^H$ as in \eqref{equ-MT2a}.
	For $k = 1,\ldots,n$, let
	\begin{eqnarray*}
		\cal V(u,v,x) & := & d\Phi_{(u,v,x)}(\partial_v), \\
		\cal X_k(u,v,x) & := & d\Phi_{(u,v,x)}(\partial_k), \\
		\cal Z(u,v,x) & := & d\Phi_{(u,v,x)}(\partial_u)
	\end{eqnarray*}
	denote the push-forward vector fields. Then, since the leaves $\wt{\cal N}$ of $\wt V^\bot$ are totally geodesic, we have that 
	$\cal V(u,v,x)\in {\wt V}^\bot_{\Phi(u,v,x)}$ and $\cal X_k(u,v,x) \in {\wt V}^\bot_{\Phi(u,v,x)}$ for $k=1, \ldots , n$.
	Furthermore, along the integral curve $\gamma$ of $\wt Z$, we have
	$\cal V (u,0,0)=\wt V(\gamma(u))$ and $\cal X_k(u,0,0)=S_k (\gamma(u))$. 
	Moreover we can show
	
	\blem\label{transportlemma}
	For each $(u,v,x)\in \rr^{n+2}$, consider the geodesic 
\[
	\rr\ni t\mapsto \delta(t):=\Phi(u,tv,tx)=
	\exp_{\gamma(u)}(t  ( v  \wt V(\gamma(u)) + \bigl\langle x, \vec{S}(u)\bigr\rangle)).\]
The vector fields $t\mapsto \cal V(u,tv,tx)$ and $t\mapsto\cal X_k(u,tv,tx)$ are parallel transported along $\delta$.
	
	\elem
\bprf
For each $(u,v,x)\in \rr^{n+2}$ consider the geodesic variation $F:\rr\times (-\ve,\ve)\to \tem$,
\[F(t,s)
=
\Phi(u,t(v+s),tx)= 
 \exp_{\gamma(u)}\left(  t \big( ( v+s)  \wt V(\gamma(u)) +  \bigl\langle x, \vec{S}(u)\bigr\rangle\big)\right)\]
 of the geodesic
$ \delta(t):=F(t,0)$.
The variation vector field along $\delta$ is given as 
\[
t\mapsto 
	\frac{\partial F}{\partial s}(t,0)=
dF|_{(t,0)}\left(\frac{\partial}{\partial s}\right)
=
d\Phi_{(u,tv,tx)}( t\del_v)
=
t\ \cal V(u,tv,tx).
\]
Thus, as  the variation vector field of a  variation of $\delta(t)$ by geodesics,  $Y(t):= t \cal V(u,tv,tx)$ is a Jacobi vector field along $\delta$. Hence,
since $\delta'(t)\in \wt V_{\delta(t)}^\bot $ as well as 
$Y(t)\in \wt V_{\delta(t)}^\bot $, we have 
\[
\frac{\nabla^2}{dt^2}Y(t)
=
\wt R (\delta'(t), Y(t))\delta'(t)
=0,
\] 
by the curvature properties of a pp-wave. We can apply Lemma~\ref{Lemma:Jacobifields on flat manifolds} and obtain that $t\mapsto \cal V(u,tv,tx)$ is parallel transported along the geodesic $\delta$.
 The same argument, using the geodesic variation 
 \[F_k(t,s):=
 \Phi(u,tv, tx +s \e_k),
 \]
 shows that the $\cal X_k$ are parallel transported along $\delta$.
\eprf
Recall that for $t=0$ we know that  $\cal V(u,0,0)=\wt V(\gamma(u))$ and $\cal X_k(u,0,0)=S_k(\gamma(u))$. On the one hand,  since $\wt V$ is parallel, in particular along $\delta$, this implies that $\cal V(u,tv,tx)=\wt V(\delta(t))$ and hence $\wt V=\cal V$ everywhere on $\tem$. On the other hand, it implies that
\be
\tg_{\phi(u,v,x)} (\cal V(u,v,x),\cal V(u,v,x))\ =\ \tg_{\gamma(u)} (\wt V(\gamma(u)),\wt V(\gamma(u)))&=&0,
\\
\tg_{\phi(u,v,x)} (\cal V(u,v,x),\cal X_k(u,v,x))\ =\ \tg_{\gamma(u)} (\wt V(\gamma(u)),S_k(\gamma(u)))&=&0,
\\
\tg_{\phi(u,v,x)} (\cal X_i(u,v,x),\cal X_j(u,v,x))\ =\ \tg_{\gamma(u)} ( S_i(\gamma(u)),S_j(\gamma(u)))&=&\delta_{ij},
\ee
and thus	$$
		\Phi^*\wt{g}(\partial_v,\partial_v) = 0, \ \Phi^*\wt{g}(\partial_v, \partial_k) 
		= 0 \text{ and } \Phi^*\wt{g}(\partial_i, \partial_j) = \delta_{ij}.
	$$
	It remains to show that $\Phi^*\wt{g}(\partial_k,\partial_u) = 0$ and 
	$\Phi^*\wt{g}(\partial_v,\partial_u) = 1$. For the second equation consider
	for fixed $v \in \R$ and $x \in \R^n$ the variation
	$$
		\nu(u,s) := \Phi(u, sv, sx)
	$$
	and let $\nu_u(u,s) := \frac{\partial \nu}{\partial u}(u,s)$ and 
	$\nu_s(u,s) := \frac{\partial \nu}{\partial s}(u,s)$. Observe that $\nu_s \in {\wt V}^\bot$
	and $\cal Z(u,sv,sx) = \nu_u(u,s)$. Consequently, by  Schwarz' lemma and the parallelity of $\wt V^\bot$, 
	\begin{equation}\label{schwarz}
		\frac{\nabla}{ds} \nu_u(u,s) = \frac{\nabla}{du} \nu_s(u,s) \in {\wt V}^{\bot}.
	\end{equation}
This implies
	$$
		\frac{d}{ds}{\wt g}_{\nu(u,s)}(\nu_u(u,s), \wt V(\nu(u,s))) \equiv 0,
	$$
i.e.,  $s \mapsto {\wt g}(\nu_u(u,s), \wt V(\nu(u,s)))$ is constant and thus 
	equals its value in $s = 0$, which is
	$$
		{\wt g}_{\nu(u,0)}(\dot\gamma(u), \wt V(\gamma(u))) 
		=
		{\wt g}_{\gamma(u)}(\wt Z(\gamma(u)), \wt V(\gamma(u))) 
		\equiv 1,
	$$
since $\nu(u,0)=\gamma(u)$, $\nu_u(u,0)=\dot \gamma(u)$ and since $\gamma$ is an integral curve of $\wt Z$. This proves $\Phi^*\wt{g}(\partial_v,\partial_u) = 1$.
	
		To see $\Phi^*\wt{g}(\partial_k,\partial_u) = 0$, consider the identity
	\begin{equation}
		\label{curv-var}
		R^{\wt g}(\nu_u, \nu_s)\nu_s = \frac{\nabla}{du}\frac{\nabla}{ds}\nu_s 
		- \frac{\nabla}{ds}\frac{\nabla}{du}\nu_s.
	\end{equation}
	Since $s\mapsto \nu(u,s)$ is a geodesic for every $u \in \R$, it holds $\frac{\nabla}{ds}\nu_s = 0$. 
	Taking into account that $\nu_s \in {\wt V}^{\bot}$ we have by the definition of a pp-wave, see also Proposition \ref{pp-prop}\eqref{screen-flat2}, that $R^{\wt g}(\nu_u, \nu_s)\nu_s \in \R \cdot \wt V$ and hence \eqref{schwarz} and 
	(\ref{curv-var}) yield
	$$
		\frac{\nabla}{ds}\frac{\nabla}{ds}\nu_u(u,s) = \varphi(\nu(u,s)) \cdot \wt V(\nu(u,s))
	$$
	for some function $\varphi \in C^\infty(\wt{\cal M})$. We conclude that
	\[
	\frac{d}{ds}{\wt g}_{\nu(u,s)}(\tfrac{\nabla}{ds}\nu_u(u,s), \cal X_k(u,sv,sx))
	=
{\wt g}_{\nu(u,s)}(\tfrac{\nabla}{ds}\nu_u(u,s),
	\tfrac{\nabla \cal X_k}{ds}(u,sv,sx))
	=0,
\]
	because of Lemma~\ref{transportlemma}. Hence,
	$$
		s \mapsto {\wt g}_{\nu(u,s)}(\tfrac{\nabla}{ds}\nu_u(u,s), \cal X_k(u,sv,sx))
	$$
	is constant and equals its value in $s = 0$. But for $s=0$ we have
	$$
		{\wt g}_{\nu(u,0)}(\tfrac{\nabla}{ds}\nu_u(u,0), \cal X_k(u,0,0)) = 0,
	$$
	since $\cal X_k(u,0,0) = S_k(\gamma(u))$, and
	\be
		\frac{\nabla}{ds}\nu_u(u,0) \ 
		=\  \frac{\nabla}{du}\nu_s(u,0)
		&=&
		\frac{\nabla}{du}\left( \frac{d}{ds}\left( \exp_{\gamma(u)}(s\big(v\wt V (\gamma(u))
		 +  \bigl\langle x, \vec{S}(\gamma(u))\bigr\rangle\big)\right)|_{s=0}\right)
		 \\
		 &
		=
		& \frac{\nabla}{du}\left(v \cdot \wt V(\gamma(u)) + \sumk x^kS_k(\gamma(u))\right) 		\\
		&=& \sum_{k,l=1}\delta_{kl} x^k\alpha^l(\dot\gamma(u)) \wt V(\gamma(u)),
	\ee 
	by  the Schwarz lemma, the parallelity of $\wt V$ and the property of $S_k$. Note that we use here that 
	$ \frac{\nabla}{du}( x^kS_k)= x^k\frac{\nabla}{du} S_k$ since the $x^k$ are constant along the curve $\gamma(u)$.
	Hence,
\[{\wt g}_{\nu(u,s)}(\tfrac{\nabla}{ds}\nu_u(u,s), \cal X_k(u,sv,sx))\equiv 0.\]
	Finally, this and  Lemma~\ref{transportlemma} imply that
	\[
		\frac{d}{ds}{\wt g}_{\nu(u,s)}(\nu_u(u,s), \cal X_k(u,sv,sx))
		= 
		{\wt g}_{\nu(u,s)}(\tfrac{\nabla}{ds}\nu_u(u,s), \cal X_k(u,sv,sx))
	=0.\]
	Hence, also $s \mapsto {\wt g}_{\nu(u,s)}(\nu_u(u,s), \cal X_k(u,sv,sx))$ is constant and 
	as $\nu_u(u,0) = \dot\gamma(u)$ we obtain at $s = 0$:
	$$	
		{\wt g}_{\nu(u,0)}(\dot\gamma(u), S_k(\gamma(u)))= 		{\wt g}_{\gamma(t)}(\wt Z(\gamma(t)), S_k(\gamma(u))) = 0.
	$$
	Thus, the only non constant term in the metric $\Phi^*\tg$ on $\rr^{n+2}$ is the function $H \in C^\infty(\R \times \R^n)$  defined by
	$$
		2H := (\Phi^*\wt g)(\partial_u,\partial_u).
	$$
	This finishes the proof of the second statement of  Theorem~\ref{covertheo}.
\eprf
 

\bbem
Note that, at this stage we do not make a claim about the geodesic completeness of pp-waves satisfying the assumptions of Theorem~\ref{covertheo}. This will depend on the function $H$. We will give a sufficient condition in Lemma~\ref{completelemma} in the next section, which we will then use to establish completeness for compact pp-waves.
\ebem

\section{Completeness of compact pp-waves}
\label{compactsection}

\subsection{Proof of Theorem~\ref{MainTheo2}}
\label{thmAsection}
Theorem ~\ref{MainTheo2} will follow from Theorem~\ref{covertheo}.
Since $\cal M$ is assumed to be compact, every screen vector field is complete. Hence, we  have to verify that a compact pp-wave satisfies the other assumption of Theorem~\ref{covertheo}:

\begin{theo}
	\label{MainTheo1}
	Let $(\cal M, g)$ be a compact pp-wave with parallel null vector field 
	$V $. Then the maximal  geodesics  along the leaves of the parallel distribution $V^\bot$  are defined on $\rr$.
\end{theo}


	\bprf 
	 Again,  by a tilde we shall denote the lift of an object to the 
	universal cover $\tem$ of $\cal M$.
				First, let $Z$ be a screen vector field and $\SS$ the corresponding screen 
	distribution on $\cal M$. As $\cal M$ is compact, $Z$ is complete and we can apply 
	Proposition~\ref{PropUnivSplit} to $Z$ and  $\eta:=g(V,.)$ and obtain that the universal 
	cover $\tem$ of $\cal M$  is diffeomorphic to $\R \times \widetilde{\cal N}$, where 
	$\widetilde{\cal N}$ is the universal cover of a leaf $\cal N$ of the distribution $V^\bot$ 
	of $\cal M$. Clearly, $\widetilde{\cal N}$ is also a leaf of the distribution 
	$\widetilde V^\bot $ on $\tem$. Since $(\M,g)$ is a pp-wave, the lift $\wt{\SS}$ of the screen distribution $\SS$ comes with a global frame field $S_i\in \Gamma (\wt{\SS})$, $i=1, \ldots, n$, on $\tem$ satisfying the relations (see Proposition~\ref{PropTrivHol}) 
	\begin{equation}\label{nabsi}\tnab _XS_i=\alpha_i(X)\wt V.\end{equation} 
	Note that the $S_i$ are not necessarily lifts of global vector fields on the compact $\M$, however we will show that they are complete.
	
To this end, consider the  Riemannian metric $h$ on $\cal M$ defined by the original screen distribution $\SS$ on $\cal M$ via
	\[h(V,V)=h(Z,Z)=1,\ h(V,Z)=h(V,X)=h(Z,X)=0,\ h(X,Y)=g(X,Y),\]
	for all $X,Y\in \Gamma(\SS)$. As a Riemannian metric on a compact manifold $\cal M$ it is 
	geodesically complete, and so is its restriction to the leaves $\cal N$ of $V^\bot$, see for 
	example \cite[Exercise 10.4.28]{10}.
%
Therefore, the lifted Riemannian metric $\wt h$  on $\wt{\cal N}$ is geodesically complete. 
	Now one computes that the  vector fields $S_1, \ldots, S_n$ on $\wt{\cal N}$, which are 
	$\wt h$-orthonormal, 
	 span the 
	lifted screen $\wt{\SS}$ and satisfy equation~\eqref{nabsi}, are in fact geodesic 
	vector fields for $(\wt{\cal N}, \wt h)$. Indeed, from the Koszul formula we get
	\be 
	0\ =\ \wt g(\tnab_{S_i}S_i,X)
	&=&S_i(\wt g(S_i,X)) +\wt g([X,S_i],S_i)
	\\
	&= &S_i(\wt h(S_i,X)) +\wt h([X,S_i],S_i)\ =\ \wt h(\nabla^{\wt h}_{S_i}S_i,X),
	\ee
	for all $X\in \Gamma(T\wt{\cal N})$. Here the replacement of $\wt g$ by $\wt h$ is justified since 
	$$
		\wt g(S_i,.)|_{T\wt{\cal N}} =\wt h(S_i,.)|_{T\wt{\cal N}}.
	$$
	With $(\wt{\cal N}, \wt h)$ 
	being geodesically complete and $S_i$ being geodesic vector fields, this yields the 
	conclusion that the $S_i$ are complete vector fields.
		Now we need
	\blem \label{screenlemma}
	Let $(\M,g)$ be a pp-wave with a complete parallel null vector field $V$ and assume that there is a complete screen vector field $Z$. Then 
		there is a 
		horizontal 
	and involutive realization $\widehat{\SS}$ of the screen bundle $\widetilde{\Sigma}$ on the universal cover $\tem$, and the leaves $\wt{\cal N}$ of $\wt V^\bot$ are diffeomorphic to $\rr\times \hat{\cal S}$,
	where $\hat{\cal S}$ is a leaf of the  distribution $\widehat{\SS}$. In particular, $\tem$ is diffeomorphic to $\rr^2\times \hat{\cal S}$.
	\elem
	\bprf  
Since  $Z$ is complete we can apply Theorem~\ref{ThmHI1} and obtain a horizontal 
	and involutive realization $\widehat{\SS}$ of the screen bundle $\widetilde{\Sigma}$ on 
	the universal cover $\tem$ and a corresponding screen vector field 
	$\hat Z\in \Gamma(T\tem)$. Furthermore, 
	consider  the Riemannian metric $\hat h$ on $\widetilde{\cal N}$ defined by
	\[
	\hat h(\widetilde V,\widetilde V)=1,\ \hat h|_{\hat\SS\times \hat\SS} 
		= g|_{\hat\SS\times \hat\SS}, \  \hat h(\widetilde V,.)|_{\hat\SS} = 0, 
	\]
	and  $\hat\eta\in \Gamma(T^*\widetilde{\cal N})$ defined by $\hat\eta(X)=\hat h(\widetilde V,X)$. 
	Then $\hat \eta$ is closed, since $\hat\SS$ is integrable and hence
	\[d\hat\eta (X,Y)= \hat h([X,Y],\widetilde V) =0.\]
By  assumption, $V$ and hence its lift  $\widetilde{V}$ are complete vector fields. Thus we can again apply 
	Proposition~\ref{PropUnivSplit}, this time to $\widetilde{\cal N}$, $\hat\eta$ and $\widetilde V$, 
	to conclude the proof.\footnote{One can also argue in the following way:
	From Proposition~\ref{prop2} we know that $\wt V$ is a parallel vector field on  the Riemannian 
	manifold $(\widetilde{\cal N},\hat h) $ but also that $\wt V$ as a lift of the complete vector 
	field $V$ is complete. Hence, as $\wt {\cal N}$ is simply connected, the flow of $\wt V$
	separates a line $\rr$ from $\wt{ \cal N}$ with orthogonal complement being the leaves 
	$\hat {\cal S}$ of the integrable distribution $\hat \SS$, again proving the lemma.}
	\eprf

Let 	$\widehat{\SS}$ 	be a horizontal 
	and involutive screen distribution, obtained from 
Theorem~\ref{ThmHI1}, with corresponding Riemannian  metric $\hat{h}$ on $\wt{\cal N}$.
	Consider the $\hat h$-orthonormal vector fields 
	$\hat S_1, \ldots, \hat S_n \in \Gamma(\hat{\SS})$ with
	$\tnab \hat{S_i}|_{\widetilde{\cal N}} = 0$ and given by 
	$$
		\hat S_i=S_i-b_i \wt V
	$$ for some real functions $b_i \in C^\infty(\wt{\cal M})$.	According to Proposition~\ref{prop2}, $\wt V$ together with  the $\hat S_i$'s form a frame of $T\wt{\cal N}$ consisting  of $\hat h$-parallel vector 
	fields. Using Lemma~\ref{screenlemma}, for these we prove 
	\blem\label{silemma}
		The vector fields $\hat S_i$ are complete.
	\elem
	\bprf We saw that the  vector fields $S_i$ are complete, i.e., we obtain their  flows as
	\[
		\phi^i: \rr \times \wt{\cal N} \to \wt{\cal N}.
	\]
	Recall that, by Proposition~\ref{PropUnivSplit}, the leaf $\wt{\cal N}$ is diffeomorphic	to $\R \times \hat{\cal S}$ via
	$$
		\Psi: p \in \ten \longmapsto (\varphi(p), \psi_{-\varphi(p)}(p)) \in \rr \times \hat{\cal S},		
	$$
	were $\{\psi_t\}$ is the flow of $\wt V$ and $\varphi \in C^\infty(\ten)$, such that 
	$\hat h(\wt V, \cdot)|_{{\wt V}^\bot} = d\varphi$ (see the proof of Proposition~\ref{PropUnivSplit}).
	Under this diffeomorphism, the flows \{$\phi_t^i$\} are given as
	$$
		\wt{\phi}^i_t(p) := \Psi(\phi^i_t(p)) = (\nu^i_t(p), \hat{\phi}^i_t(p)),
	$$
	with
	\begin{eqnarray*}
		\nu^i_t(p) & := &\varphi(\phi^i_t(p)), \\
		\hat{\phi}^i_t(p)& := &\psi_{-\nu^i_t(p)}(\phi^i_t(p)),	
	\end{eqnarray*}
	both defined for all $t\in \rr$.
	We do now claim that $\{\hat{\phi}^i_t\}$ is the flow of $\hat S_i$. Indeed,
	on the one hand we have that
	\begin{equation}
		\label{pr-MainTheo1-eq1}
		\frac{d}{dt}\wt{\phi}^i_t(p) 
		= (\tfrac{d}{dt}\nu^i_t(p), \tfrac{d}{dt}\hat{\phi}^i_t(p))
		= (\tfrac{d}{dt}\nu^i_t(p), 0) + (0, \tfrac{d}{dt}\hat{\phi}^i_t(p)).
	\end{equation}
On the other hand we compute, using Lemma~\ref{screenlemma}, the chain rule and the linearity of the differential
\begin{equation}
		\tfrac{d}{dt}\wt{\phi}^i_t(p) 
		= 
			d\Psi_{\phi^i_t(p)}(S_i(\phi^i_t(p)))  
			= 
			d\Psi_{\phi^i_t(p)}(\hat S_i(\phi^i_t(p))) + b_i(\phi^i_t(p))d\Psi_{\phi^i_t(p)}(\wt V(\phi^i_t(p))) \label{flows1} .
	\end{equation}
 Temporarily 	denoting by 
 $\{\xi_t^i\}$  the flow of $\hat S_i$,
	for the first term  we get
	\begin{eqnarray*}
	d\Psi_{\phi^i_t(p)}(\hat S_i(\phi^i_t(p)))
	\ =\ 
	 \left.\tfrac{d}{d\tau}\Psi(\xi_\tau^i(\phi^i_t(p)))\right|_{\tau = 0} 
	 &=&
	 \left.\tfrac{d}{d\tau}\left( 
	 \nu^i_t(p), \psi_{-\nu_t(p)}\circ \xi^i_\tau(\phi^i_t(p))\right)\right|_{\tau = 0} 	 
	 	 \\
	 &=&
	 \left.\tfrac{d}{d\tau}\left( 
	 \nu^i_t(p), \xi^i_\tau  \circ \psi_{-\nu_t(p)} (\phi^i_t(p)) \right)\right|_{\tau = 0} 	 
	 	 \\
	 &=&
	 \left.\tfrac{d}{d\tau}\left( 
	 \nu^i_t(p), \xi^i_\tau  \left( \hat{\phi}^i_t(p) \right)\right)\right|_{\tau = 0} 
	 \\
	 & = & (0, \hat S_i(\hat{\phi}^i_t))\in \rr\oplus T\hat{\cal S},	 
	 \end{eqnarray*}
	in which we were allowed to commute the flows $\xi^i_\tau$ and $\psi^i_s$  because of  $[\wt V, \hat S_i] = 0$. 
	For the second term in \eqref{flows1} we recall that 
	\[ \tfrac{d}{d\tau}\varphi(\psi_\tau(\phi^i_t(p)))|_{\tau = 0}=d\varphi_{\phi^i_t(p)}(\wt V|_{\phi^i_t(p)})=
	\hat h(\wt V,\wt V )|_{\phi^i_t(p)}\equiv1,\]
	which implies $\varphi(\psi_\tau(\phi^i_t(p))=\tau+c$ for a constant $c$. Hence, we get
	\be
	d\Psi_{\phi^i_t(p)}(\wt V(\phi^i_t(p)))			
		&		 =  & \left.\tfrac{d}{d\tau}\Psi(\psi_\tau(\phi^i_t(p)))\right|_{\tau = 0} 
		\\
	&=&
	\left.\tfrac{d}{d\tau}\left( \vf(\psi_\tau(\phi^i_t(p))), \psi_{-\vf(\psi_\tau(\phi^i_t(p)))}(\psi_\tau(\phi^i_t(p)))\right)\right|_{\tau=0}\\
		&		= &
	\left.\tfrac{d}{d\tau}\left( \tau+c, 	
		 \psi_{-c}\circ \psi_{-\tau}(\psi_\tau(\phi^i_t(p)))\right)\right|_{\tau=0}\\
	&=&	   (1,0)\in \rr\oplus T\hat{\cal S} 
	\ee
	Both computations show that \eqref{flows1} becomes
	\[
			\tfrac{d}{dt}\wt{\phi}^i_t(p)
			= \left(  b_i(\phi^i_t(p)) , \hat S_i(\hat{\phi}^i_t)\right)\in \rr\oplus T\hat{\cal S},
			\]
	which, together with \eqref{pr-MainTheo1-eq1}, shows that
	$\tfrac{d}{dt}\hat{\phi}^i_t(p) = \hat S_i(\hat{\phi}^i_t(p))$. Hence, $\hat{\phi}^i_t$ is the flow of $\hat S_i$ which is defined on $\rr$. This proves the lemma.
	\eprf

Thus, applying Lemma~\ref{silemma} we \edremoved{can proceed  as in the proof of Theorem~\ref{covertheo} and} obtain that,  on the simply connected Riemannian manifold $(\wt{\cal N}, \hat h)$, we have
 a frame  $(\wt V, \hat{S}_1, \ldots, \hat{S}_n)$ of complete vector fields which are $\nabla^{\hat h}$-parallel and 
$\hat h$-orthonormal. \edrevised{Then it is obvious that $\nabla^{\hat h}$ is a complete connection}, implying that  $(\wt{\cal N}, \hat h)$ is geodesically complete.  
	On $\wt {\cal N}$	the frame $(\wt V, \hat{S}_1, \ldots, \hat{S}_n)$ is parallel for both,
	the Levi-Civita connections $\wt{\nabla}^g $  of $\wt{g}$ and $\nabla^{\hat h}$ of $\hat h$, the 
	connections  are equal, and whence, the leaf $\wt{\cal N}$ of $\wt{V}^\bot $ is geodesically complete 
	for the  metric $\wt g$. Hence, the leaves $\cal N$ of $V^\bot$ on $\cal M$ are 
	geodesically complete for  $g$.
\eprf
\edrevised{
\bbem
Note that one of the key steps in the proof is Lemma \ref{silemma}. It ensures  that the completeness of the vector fields $S_i$ on $\wt{\M}$, which are complete as lifts of vector fields on the compact manifold $\M$, implies the completeness  of the vector fields $\hat S_i =S_i-b_i \wt V$, {\em despite the fact the functions $b_i$  do not come from functions on the compact manifold $\M$} and, for example, may be unbounded. However the splitting results enable us to relate the flow of the $\hat S_i$'s to the flow of the $S_i$'s in a way that their completeness, and as a consequence, the completeness of $\wt g$ and thus of $g$ follow.
\ebem
}



\subsection{Proof of Theorem~\ref{MainTheo3} and Corollary \ref{ricfolg}}
\label{thmBsection}
Finally, the proof of  Theorem~\ref{MainTheo3} is based on Theorem~\ref{MainTheo2} and a version of results by Candela et al.~\cite{candela-flores-sanchez03} adapted to our situation\footnote{In fact, during the preparation of the paper we learned that   Lemma \ref{completelemma} follows from  stronger results by Candela et al.~\cite[Theorems 1 and 2]{candela-romero-sanchez13}. However, for the sake of being self-contained we include a  proof of the lemma. For further results and comments see \cite{candela-romero-sanchez13plane,sanchez13}}:
\blem\label{completelemma}
The pp-wave metric on $\rr^{n+2}$ in standard form
\[g^H=2du (dv +H(u,x^1, \ldots , x^n) du) +\delta_{ij}dx^idx^j\] is geodesically complete if all second $x^i$-derivatives of $H$ are bounded, $\left| \frac{\del^2H}{\del x^i\del x^j}\right|\le c$ for a positive constant $c$ and $1\le i,j\le n$.
\elem

%
%

\bprf
	By Theorem~\ref{Lem-Sanchez}, $g^H$ is  complete if every maximal solution 
	$\gamma : s \mapsto \gamma(s) \in \R^n$ of
	\begin{equation}
		\label{eqMainTheo3}
		\ddot \gamma(s) = F(s,\gamma):= \grad_{\R^n} H(s,\gamma(s))
	\end{equation}
is defined on 
	the whole real line. 
	Now, recall the following fact, see for example \cite[Theorem 2.17]{teschl12}:
	{\em Let $F:\rr \times \rr^{2n} \to \rr^n$ be globally Lipschitz on every set of the 
	form $I \times \rr^{2n}$, where $I$ is a closed interval, then, for every initial 
	value $(t_0,x_0,x_1)\in \rr\times \rr^{2n}$ 
	there is a solution $x:\rr \to \rr^n$ of the initial value problem 
	$\ddot{x}=F(t,x,\dot x)$ with $x(t_0)=x_0$ and $\dot x(t_0)=x_1$}. 
	We thus have to show, that the function $F : [a,b] \times \R^{2n} \lra \R^n$  with $F(s,x,y)=F(s,x)$ defined in 
	(\ref{eqMainTheo3}) is Lipschitz for arbitrary $a,b \in \R$. Clearly, by the mean value theorem for functions from $\rr^n$ to $\rr^n$,  if every
	partial derivative of $F$ is bounded, then $F$ is Lipschitz. But every partial derivative in the second argument of 
	$F = (F_1, \ldots, F_n)$ is given by
	$$
		\frac{\partial F_i}{\partial x_j}(t,x) = \frac{\partial}{\partial x_j}\left(\frac{\partial H}{\partial x_i}\right)(t,x),
	$$
	and thus bounded by assumption.
	We conclude that $F$ must be Lipschitz on every set $[a,b] \times \R^n$ 
	which guarantees that the maximal solutions $\gamma$ of (\ref{eqMainTheo3}) are
	defined on $\R$.
	\eprf
	\bbem\label{EKremark}
In regard to the Ehlers-Kundt problem mentioned in the introduction,	 Lemma~\ref{completelemma} provides us with many examples of pp-waves that are not plane waves. Again, these examples cannot be Ricci-flat, since harmonic functions do not have 
bounded second derivatives unless 
they are quadratic and thus a pp-wave.
	\ebem
	
	The proof of Theorem~\ref{MainTheo3} will follow from
	\blem 
	\label{hess-lemma}
	Let $(\M,g)$ be a compact pp-wave and let $g^H=2du (dv +H du) +\delta_{ij}dx^idx^j$ be the metric 
on	the  universal cover $\rr^{n+2}$ of $\M$ that is globally isometric to the 
	 lift of~$g$. Then all second covariant derivatives  of $H$ in $x^i$-directions are bounded, 
	\[0 \le \del_i\del_j H\le c,\ \text{ for all $i,j=1, \ldots n$.}\]
	\elem

	\bprf
	Let $\phi : (\R^{n + 2}, g^H) \lra (\cal M, g)$ denote the isometric universal covering map from Theorem~\ref{MainTheo2}.
	Let $Z \in \Gamma(T\cal M)$ be an arbitrarily chosen screen
	vector field and $\wt Z \in \Gamma(T\wt{\cal M})$ its pullback to $\wt{\cal M}$. 
	Note that we have particularly shown in Theorem~\ref{MainTheo2} that  $g(d\phi(\partial_u), V) = 1$, and hence we have that  
	\belabel{dphiz}d\phi(\partial_u) = Z + \sumi b_iS_i + c  V\end{equation}for smooth functions 
	$b_i,c \in C^\infty(\wt{\cal M})$ and $S_i$ a basis of the screen distribution corresponding to $Z$.
Now we define 	a symmetric $(0,2)$-tensor field on $\cal M$ as
	$$
	Q (X,Y) := R (X,Z,Z,Y).
	$$
Since $\cal M$ is compact, the  function  $\overline{g}(Q,Q)$, where $\overline{g}$ denotes the metric induced by $g$ on $(0,2)$-tensor fields,  is bounded, i.e., $-C^2<\overline{g}(Q,Q) < C^2$ for some constant $C \in \R^+$. Computing $\overline{g}(Q,Q)$ in a frame $V,Z,E_1,\ldots , E_n$ with $E_i$  an orthonormal frame of the screen defined by $Z$, the obvious equation $Q(V,.)=0$ gives us
\[
\overline{g}(Q,Q)
\ =\ \sum_{i,j=1}^n Q(E_i,E_j)^2 \\
 \ =\ 
\sum_{i,j=1}^n R (E_i,Z,Z,E_j)^2,
\]
so we have in fact that $0\le \overline{g}(Q,Q) < C^2$.
	
	Pulling back $Q$ to the universal cover $(\rr^{n+2}, g^H)$ by the  isometric covering map $\phi$,  using \eqref{dphiz}, \eqref{curv} and \eqref{screen-flat}, we get that $\phi^*Q(\del_v,.)=0$ and 
	\be
	\phi^*Q(\del_i,\del_j)_x &
	=& R _{\phi(x)} (d\phi_x(\del_i),Z,Z, d\phi_x(\del_j)
	\\
	&	=& 
	R _{\phi(x)} (d\phi_x(\del_i),d\phi_x(\del_u),d\phi_x(\del_u), d\phi_x(\del_j))
	\\
		&=& \phi^*R_{x} (\del_i,\del_u,\del_u,\del_j)
		\\&=&
		R^{g^H}_x(\del_i,\del_u,\del_u,\del_j)
		\\
		&=&
		-\del_i\del_jH(x).
	\ee
	Hence, by using a frame $(\del_v,\del_u-H\del_v, \del_i)$ on $(\rr^{n+2},g^H)$ to compute $\overline{g^H}(\phi^*Q,\phi^*Q)$,  at each point in  $\rr^{n+2}$ we have
	$$
		C^2 > \overline{g}(Q,Q) = \overline{g^H} (\phi^*Q, \phi^*Q )= 
		\sumij 		\phi^*Q(\del_i,\del_j)^2
		=
		\sumij (\del_i\del_jH)^2,
	$$
	which shows that all $\del_i\del_j H$ are bounded. \eprf
	\bprf[Proof of Theorem~\ref{MainTheo3}]
	Let $(\M,g)$ be a compact pp-wave. 
	Because of Theorem \ref{MainTheo2}, the universal cover is isometric to a standard pp-wave $(\rr^{n+2},g^H)$, and by Lemma \ref{hess-lemma}, all $\del_i\del_j H$ are bounded. Then, 
by	 Lemma~\ref{completelemma}, $(\rr^{n+2},g^H)$ is complete, and thus 
  $(\M,g)$ itself is complete.
 \eprf
 
 Lemma \ref{hess-lemma} also provides us with a proof of Corollary~\ref{ricfolg}:
 \bprf[Proof of Corollary~\ref{ricfolg}]
Let $(\M,g)$ be a compact pp-wave and let $(\rr^{n+2}, g^H)$ be the standard pp-wave that is globally isometric to the universal cover of $(\M,g)$.  Lemma \ref{hess-lemma} tells us that the $\del_i\del_jH$ are bounded. If  $g$ is Ricci-flat, so is $g^H$, and thus $H$ is harmonic with respect to the $x^i$-directions, i.e., $\sumi\del_i^2(H)=0$. But this implies that also  $\del_i\del_jH$ is harmonic in the same sense, and thus, by the maximum principle for harmonic functions, independent of the $x^i$ components. Hence,  
\[
H=\sumij a_{ij}(u)x^ix^j+b_ix^i+c \] with $a_{ij}$, $b_i$ and $c$ functions of $u$ only, which implies that $(\M,g)$ is a plane wave. 
 \eprf
\subsection{Plane waves} Finally, we apply Theorems \ref{MainTheo2} and \ref{MainTheo3} to plane waves as defined  in Definition \ref{planedef}.


\bfolg
	An $(n+2)$-dimensional compact plane-wave is geodesically complete and its universal 
	cover is isometric to $\rr^{n+2}$	with the metric $g^H$ defined in Theorem 
	\ref{MainTheo2}, where $H(u, x) = \sumkl a_{kl}(u)x^kx^l$ for some
	$a_{kl}=a_{lk} \in C^\infty(\R)$.
\efolg


\bprf
Since plane waves are pp-waves, Theorem~\ref{MainTheo3} implies that compact plane waves are complete. Furthermore, by
 Theorem~\ref{MainTheo2}, we have for the universal covering that
 $$
 	R^{g^H}(\partial_i,\partial_u,\partial_u,\partial_j) 
 	=\hess \,H(\partial_i, \partial_j) 
 	= \partial_i(\partial_j(H))
 $$
The additional plane wave condition $\nabla  R = V^\flat \otimes Q$ implies  for the universal cover that 
 \[0=(\nabla_{\partial_k} R^{g^H})(\partial_i,\partial_u,\partial_u,\partial_j) = 
 -\del_k\del_i\del_j(H),
 \]
 since $\nabla_{\del_k}\del_u\in \del_v^\bot$. This implies that
 $$
 	H(u, x) = \sumkl a_{kl}(u)x^kx^l + \sumk b_k(u) x^k + c(u).
 $$ 
Getting rid of the linear and constant terms in this expression is achieved by a  coordinate transformation of the 
form \[
\wt v=  v -\dot{\beta}_i(u) x^i+\gamma (u)
,\ \ 
\wt x^i =x^i+\beta_i(u),\ \ 
\wt u=u
\]
where $\beta$ and $\gamma$ are obtained by integrating
\be
\ddot{\beta_i}(u)&=& - b_i(u),\\
\dot\gamma(u)&=& c(u)-\tfrac{1}{2}\sumi \dot \beta_i(u)^2 ,
\ee
with initial conditions $ \beta_i(0)=0$ and $\gamma(0)=0$.
\eprf
In view of Corollary \ref{ricfolg}, note that plane waves in standard form are Ricci flat if and only if the matrix $a_{ij}$ is trace-free.


\bbem
	If we weaken the assumption made within this paper that the null vector field $V$ is parallel, to $V$ being 
recurrent, i.e. with $\nabla V = \varphi \otimes V$, then a compact Lorentzian manifold with the
	curvature condition of a pp-wave but with such a
	recurrent vector field\footnote{In \cite{leistner05c} we called these Lorentzian manifolds  
	\textit{pr-waves} for {\em plane fronted with recurrent rays}.} 
	is not necessarily  complete. This means \textit{Theorem~\ref{MainTheo3} cannot
	be generalized to compact Lorentzian manifolds with the curvature conditions of a pp-wave but with 
	recurrent null vector field}. Even if $\varphi(X) = 0$ for all $X \in V^\bot$ such that the 1-form $g(V, \cdot)$ is still
	closed, the result is false in general, as the following example shows.
\ebem

\bbsp\label{incompex}
	Consider $\wt{\cal M} := \R^{n + 2}$ endowed with the metric
	$$
		\wt g_{(u,v,x_1, \ldots, x_n)} := 2dudv - 2\left(
		\sin(v) - \sumi a_i (\cos(x_i) - 1)\right)
		du^2 + \sumi dx_i^2,
	$$
	for constants $a_i$. Being $2\pi$-periodic, the metric $\wt g$ 
	descends to a metric $g$ on the torus $\mathbb T^{n + 2} := \R^{n + 2}/2\pi\Z^{n + 2}$. The
	inextensible (transversal) geodesic 
	$$
	{\wt \gamma}(t) := (\ln(t), 0, \ldots, 0)
	$$
	then defines an 
	inextensible geodesic $\gamma : (0, \infty) \lra \mathbb T^{n + 2}$ on the compact Lorentz manifold
	$(\mathbb T^{n + 2}, g)$ by $\gamma(t) := \pi(\wt{\gamma}(t))$, with $\pi : \R^{n + 2} \lra \mathbb T^{n + 2}$
	denoting the canonical projection. For $a_i=0$ this is a version of the Clifton-Pohl torus. See also results by S\'{a}nchez \cite{sanchez97} on (incomplete) Lorentzian $2$-tori.
\ebsp

However, we do not know, if at least the geodesics along the leaves of $V^\bot$ are all complete.
For the case $\varphi(X) = 0$ for all $X \in V^\bot$ our proofs seem to be adaptable to this
situation since, in this case, $V^\flat$ is still closed.

\ednew{
\subsection{Compact indecomposable Lorentzian locally symmetric spaces} 
\label{symsec}
Before we consider {\em locally} symmetric spaces and give the     proof of Corollary~\ref{symcol}, we recall some facts about indecomposable Lorentzian (globally) symmetric spaces. First 
recall from the introduction that a Lorentzian manifold is {\em indecomposable} if it is not locally isometric to a semi-Riemannian product.  
Now there is the following dichotomy: The transvection group of a simply connected,  indecomposable Lorentzian symmetric space $(\M,g)$ is either semisimple or solvable \cite[Theorems~2 and~3]{cahen-wallach70}. If the transvection group is semisimple, 
 $(\M,g)$ is either of dimension $2$, in which case it has constant sectional curvature, or has irreducible isotropy \cite[Proposition~1]{cahen-wallach70}. 
Isotropy irreducible semi-Riemannian symmetric spaces were classified by Berger \cite[Tableau II]{berger57}. This classification applied to Lorentzian signature implies   that  $(\M,g)$ is either a de~Sitter space or the universal cover of  an anti-{de~Sitter} space and thus has constant sectional curvature (see \cite[Corollary 1.5]{olmos-discala01} for a more direct proof). In contradistinction, when the transvection group is solvable, $(\M,g)$ is a {\em Cahen-Wallach space} \cite[Theorem 5]{cahen-wallach70}, i.e., $\M=\rr^{n+2}$ and 
 \[g=g^S=2du (dv+ S_{ij}x^ix^j \, du)+\delta_{ij} dx^i\,dx^j,
 \] with a {\em constant} symmetric $n\times n$-matrix $S=\left(S_{ij}\right)_{i,j=1}^n$ that is different from the zero matrix. Having a global parallel null vector field $\partial_v$ and satisfying the curvature condition \eqref{screen-flat}, a Cahen-Wallach space is a special case of a pp-wave. 
 
\bprf[Proof of Corollary~\ref{symcol}]
 Let $(\M,g)$ be a compact  {\em locally symmetric} Lorentz\-ian manifold, which is indecomposable.  Then, as a locally symmetric space, $(\M,g)$ is locally isometric to a certain symmetric space $(\M_0,g_0)$ (see for example \cite[p.~252]{ko-no2}).  Since $(\M,g)$ is indecomposable, there is a point without a neighbourhood on which $g$ is a product metric. Hence, $(\M_0,g_0)$ must be  one of  the 
  aforementioned indecomposable  symmetric spaces, i.e., $(\M,g)$ is locally isometric to either a space of constant curvature or to a Cahen-Wallach space. In the first case $(\M,g)$ is geodesically complete by Klingler's result \cite{klingler96}.
 Hence we may assume that $(\M,g)$  is locally isometric to a Cahen-Wallach space with local coordinates $(v,x^i,u)$ and  defined by a symmetric matrix $S$. This implies that $(\M,g)$ admits a null line bundle $\cal V$ that is invariant under parallel transport. At a point $p\in \M$ the fibre of $\cal V$ is spanned by the value of the parallel coordinate vector field $\del_v$ at $p$. Moreover, in a basis $(\del_v,\del_i,\del_u-S_{ij}x^ix^j \del_v)|_p$,  an element of the holonomy group $\Hol_p(\M,g)$ at $p$ is of the form 
 \[h=\begin{pmatrix} a& *&* \\0&A&* \\0&0&a^{-1}\end{pmatrix},\]
 with $a\in \rr^*$ and $A\in \mathbf{O}(n)$ (see \cite{baum-laerz-leistner12}).
 Since $(\M,g)$ is locally symmetric, i.e. $\nabla R=0$,  the curvature tensor $R$ is invariant under the action of the full holonomy group,
$h\cdot R=R$,
 for all $h\in \Hol_p(\M,g)$. Hence,  by formula \eqref{curv},  for each pair of indices $i,j=1, \ldots n$  we get that
 \[
 S_{ij}\ \del_v
 =
 R(\del_i,\del_u)\del_j
 =
(h\cdot R)(\del_i,\del_u)\del_j
=
h(R(h^{-1}\del_i,h^{-1} \del_u)h^{-1}\del_j)
=
a^2 A_{i}^{~k}S_{kl} A_{j}^{~l}\ \del_v,
 \]
at $p$. This is nothing else than $S=a^2A^\top S A$. With $S\not=0$, it implies that $a^2=1$. This shows that the holonomy group acts on the fibre of $\cal V$ by $\pm 1$. Hence,   the time-orientable cover of $(\M,g)$, which is still compact,  admits global parallel null vector field and therefore satisfies our definition of a pp-wave. Then Theorem \ref{MainTheo3} applies, yielding that the time-orientable cover and therefore $(\M,g)$ itself are  geodesically complete.
\eprf
Note that Corollary \ref{symcol} implies that a compact locally symmetric Lorentzian manifold that is a {\em global} product of an indecomposable Lorentzian manifold with a Riemannian manifold is complete. However, 
 our proof does not immediately generalise to arbitrary decomposable (in the above local sense) compact locally symmetric Lorentzian manifolds. We believe that one can obtain a proof in the general case  by using the local de Rham and Wu decomposition theorems, but this requires to overcome some technical difficulties and we postpone this to future work.
 }



\bibliography{geobib}
\end{document}